\documentclass[12pt]{amsart}
\usepackage{amsmath,amssymb,amsbsy,amsfonts,latexsym,amsopn,amstext,cite,
                                               amsxtra,euscript,amscd,bm}
\usepackage{url}

\usepackage{mathrsfs}

\usepackage[colorlinks,linkcolor=blue,anchorcolor=blue,citecolor=blue,backref=page]{hyperref}
\usepackage{color}
\usepackage{graphics,epsfig}
\usepackage{graphicx}
\usepackage{float}
\usepackage{epstopdf}
\hypersetup{breaklinks=true}

\usepackage[norefs,nocites]{refcheck}
\def\ge {\geqslant}
\def\le {\leqslant}

\usepackage{bibentry}

\usepackage[english]{babel}
\usepackage{mathtools}
\usepackage{todonotes}

\begin{document}

\newtheorem{theorem}{Theorem}
\newtheorem{lemma}[theorem]{Lemma}
\newtheorem{claim}[theorem]{Claim}
\newtheorem{cor}[theorem]{Corollary}
\newtheorem{conj}[theorem]{Conjecture}
\newtheorem{prop}[theorem]{Proposition}
\newtheorem{definition}[theorem]{Definition}
\newtheorem{question}[theorem]{Question}
\newtheorem{example}[theorem]{Example}
\newcommand{\hh}{{{\mathrm h}}}
\newtheorem{remark}[theorem]{Remark}

\numberwithin{equation}{section}
\numberwithin{theorem}{section}
\numberwithin{table}{section}
\numberwithin{figure}{section}

\def\sssum{\mathop{\sum\!\sum\!\sum}}
\def\ssum{\mathop{\sum\ldots \sum}}
\def\iint{\mathop{\int\ldots \int}}

\newcommand{\diam}{\operatorname{diam}}

\def\squareforqed{\hbox{\rlap{$\sqcap$}$\sqcup$}}
\def\qed{\ifmmode\squareforqed\else{\unskip\nobreak\hfil
\penalty50\hskip1em \nobreak\hfil\squareforqed
\parfillskip=0pt\finalhyphendemerits=0\endgraf}\fi}

\newfont{\teneufm}{eufm10}
\newfont{\seveneufm}{eufm7}
\newfont{\fiveeufm}{eufm5}
%
%
\newfam\eufmfam
     \textfont\eufmfam=\teneufm
\scriptfont\eufmfam=\seveneufm
     \scriptscriptfont\eufmfam=\fiveeufm
%
%
\def\frak#1{{\fam\eufmfam\relax#1}}

\newcommand{\bflambda}{{\boldsymbol{\lambda}}}
\newcommand{\bfmu}{{\boldsymbol{\mu}}}
\newcommand{\bfxi}{{\boldsymbol{\xi}}}
\newcommand{\bfrho}{{\boldsymbol{\rho}}}

\def\eps{\varepsilon}

\def\fK{\mathfrak K}
\def\fT{\mathfrak{T}}

\def\fI{{\mathfrak I}}
\def\fA{{\mathfrak A}}
\def\fB{{\mathfrak B}}
\def\fC{{\mathfrak C}}
\def\fI{{\mathfrak I}}
\def\fM{{\mathfrak M}}
\def\fQ{{\mathfrak  Q}}
\def\fS{{\mathfrak  S}}
\def\fU{{\mathfrak U}}

\def\T {\mathsf {T}}
\def\Tor{\mathsf{T}_d}
\def\Tore{\widetilde{\mathrm{T}}_{d} }

\def\sM {\mathsf {M}}

\def\Kmnd{\cK_d(m,n)}
\def\Kmnp{\cK_p(m,n)}
\def\Kmnq{\cK_q(m,n)}

\def \balpha{\bm{\alpha}}
\def \bbeta{\bm{\beta}}
\def \bgamma{\bm{\gamma}}
\def \bdelta{\bm{\delta}}
\def \blambda{\bm{\lambda}}
\def \bchi{\bm{\chi}}
\def \bphi{\bm{\varphi}}
\def \bpsi{\bm{\psi}}
\def \bnu{\bm{\nu}}
\def \bomega{\bm{\omega}}

\def \bxi{\bm{\xi}}

\def \bell{\bm{\ell}}

\def\eqref#1{(\ref{#1})}

\def\vec#1{\mathbf{#1}}

\newcommand{\abs}[1]{\left| #1 \right|}

\def\Zq{\mathbb{Z}_q}
\def\Zqx{\mathbb{Z}_q^*}
\def\Zd{\mathbb{Z}_d}
\def\Zdx{\mathbb{Z}_d^*}
\def\Zf{\mathbb{Z}_f}
\def\Zfx{\mathbb{Z}_f^*}
\def\Zp{\mathbb{Z}_p}
\def\Zpx{\mathbb{Z}_p^*}
\def\cM{\mathcal M}
\def\cE{\mathcal E}
\def\cH{\mathcal H}

\def\sfB{\mathsf {B}}
\def\sfC{\mathsf {C}}
\def\L{\mathsf {L}}
\def\FF{\mathsf {F}}

\def\sD {\mathscr{D}}
\def\sE {\mathscr{E}}
\def\sS {\mathscr{S}}

\def\cA{{\mathcal A}}
\def\cB{{\mathcal B}}
\def\cC{{\mathcal C}}
\def\cD{{\mathcal D}}
\def\cE{{\mathcal E}}
\def\cF{{\mathcal F}}
\def\cG{{\mathcal G}}
\def\cH{{\mathcal H}}
\def\cI{{\mathcal I}}
\def\cJ{{\mathcal J}}
\def\cK{{\mathcal K}}
\def\cL{{\mathcal L}}
\def\cM{{\mathcal M}}
\def\cN{{\mathcal N}}
\def\cO{{\mathcal O}}
\def\cP{{\mathcal P}}
\def\cQ{{\mathcal Q}}
\def\cR{{\mathcal R}}
\def\cS{{\mathcal S}}
\def\cT{{\mathcal T}}
\def\cU{{\mathcal U}}
\def\cV{{\mathcal V}}
\def\cW{{\mathcal W}}
\def\cX{{\mathcal X}}
\def\cY{{\mathcal Y}}
\def\cZ{{\mathcal Z}}
\newcommand{\rmod}[1]{\: \mbox{mod} \: #1}

\def\cg{{\mathcal g}}

\def\vr{\mathbf r}
\def\vx{\mathbf x}
\def\va{\mathbf a}
\def\vb{\mathbf b}
\def\vc{\mathbf c}
\def\vh{\mathbf h}
\def\vk{\mathbf k}
\def\vm{\mathbf m}
\def\vz{\mathbf z}
\def\vu{\mathbf u}
\def\vv{\mathbf v}

\def\e{{\mathbf{\,e}}}
\def\ep{{\mathbf{\,e}}_p}
\def\eq{{\mathbf{\,e}}_q}

\def\Tr{{\mathrm{Tr}}}
\def\Nm{{\mathrm{Nm}}}

 \def\SS{{\mathbf{S}}}

\def\lcm{{\mathrm{lcm}}}

 \def\0{{\mathbf{0}}}

\def\({\left(}
\def\){\right)}
\def\l|{\left|}
\def\r|{\right|}
\def\fl#1{\left\lfloor#1\right\rfloor}
\def\rf#1{\left\lceil#1\right\rceil}
\def\sumstar#1{\mathop{\sum\vphantom|^{\!\!*}\,}_{#1}}

\def\mand{\qquad \mbox{and} \qquad}

\def\tblue#1{\begin{color}{blue}{{#1}}\end{color}}

%
%
%
%
%
%




\hyphenation{re-pub-lished}

\mathsurround=1pt

\def\bfdefault{b}

\def \F{{\mathbb F}}
\def \K{{\mathbb K}}
\def \N{{\mathbb N}}
\def \Z{{\mathbb Z}}
\def \Q{{\mathbb Q}}
\def \R{{\mathbb R}}
\def \C{{\mathbb C}}
\def\Fp{\F_p}
\def \fp{\Fp^*}

 \def \xbar{\overline x}

\title[Large Values of Weyl Sums]
{On Large Values of Weyl Sums}  

 \author[C. Chen] {Changhao Chen}

\address{Department of Pure Mathematics, University of New South Wales,
Sydney, NSW 2052, Australia}
\email{changhao.chenm@gmail.com}

 \author[I. E. Shparlinski] {Igor E. Shparlinski}

\address{Department of Pure Mathematics, University of New South Wales,
Sydney, NSW 2052, Australia}
\email{igor.shparlinski@unsw.edu.au}

\begin{abstract}  A special case of the Menshov--Rademacher theorem implies 
for almost all  polynomials $x_1Z+\ldots +x_d Z^{d} \in \R[Z]$ of degree $d$
for  the Weyl sums satisfy  the  upper bound 
$$
\left| \sum_{n=1}^{N}\exp\(2\pi i \(x_1 n+\ldots +x_d n^{d}\)\) \right| \leqslant N^{1/2+o(1)}, \qquad N\to \infty. 
$$
Here we investigate  the exceptional sets of coefficients $(x_1, \ldots, x_d)$ with large values of Weyl sums
for infinitely many $N$, and show that  in terms 
of the Baire categories and Hausdorff dimension they are quite massive, in particular of positive 
Hausdorff dimension in any fixed cube inside of  $[0,1]^d$. 
We also use a different technique to give similar results for sums with just one monomial $xn^d$. 
We apply these results to show that the set  of poorly distributed modulo one
polynomials   is rather massive as well. 
\end{abstract}


\maketitle

\tableofcontents

\section{Introduction}

\subsection{Motivation}

Here we consider a new type of problems of metric number theory where the vectors 
of real numbers are classified by the size of the corresponding Weyl sums given by~\eqref{eq:WeylSum} 
below,  rather than by their Diophantine approximation properties as in the classical settings, see~\cite{Bug,Harm}.   

Clearly
both points of view are ultimately related and  operated in similar notions such as the Lebesgue measure and Hausdorff dimension.
They are also both related to the question of uniformity of distribution  modulo one of fractional parts of real polynomials.  
However,  our study of sets of large Weyl sums also uses several new ideas and technics.  We believe that these ideas 
and concrete results on such  a very powerful and versatile tool as exponential sums can find applications
to other problems.  In particular, in Section~\ref{sec:ud mod1} below we give one of such applications and show that the set 
of polynomials which are poorly distributed modulo one is rather massive (in fact, our results are quantitative and thus more precise). 

In problems of this kind, the case $d\ge 3$ is much harder than the case $d=2$. The main reason is that   Lemma~\ref{lem:Gauss} below, 
giving an exact size of Gauss sums, 
which we have  for the case $d=2$,  does not   in general have any analogues for $d\ge 3$, see also  Remark~\ref{rem:why1} below.

\subsection{Set-up and background}
We now  describe our main objects of study. 

For an integer $d \geqslant 2$, let $\Tor = (\R/\Z)^d$ denote  the  $d$-dimensional unit torus. 

For  a vector $\vx = (x_1, \ldots, x_d)\in \Tor$ and $N \in\N$, we consider the exponential   
sums
\begin{equation}
\label{eq:WeylSum} 
S_d(\vx; N)=\sum_{n=1}^{N}\e\(x_1 n+\ldots +x_d n^{d} \), 
\end{equation}
which  are commonly called {\it Weyl sums\/}, where   throughout  the paper we denote $\e(x)=\exp(2 \pi i x)$. 
From the {\it Parseval identity\/}   
$$
\int_{\Tor} |S_d(\vx; N)|^{2} d \vx= N
$$
one immediately concludes that for any fixed $\alpha> 0$ the set of 
$\vx \in \Tor$ with $|S_d(\vx; N)|  \geqslant N^{\alpha}$ is of {\it Lebesgue measure\/} at most $N^{1-2\alpha}$, which is nontrivial when $1/2<\alpha<1$.

Furthermore, from  the {\it Vinogradov mean value theorem\/}, in the currently known form
$$
\int_{\Tor} |S_d(\vx; N)|^{2s(d)}d\vx \leqslant  N^{s(d)+o(1)}\qquad \text{as} \ N \to \infty,
$$
where $s(d) = d(d+1)/2$, due to Bourgain, Demeter and Guth~\cite{BDG} (for $d \geqslant 4$) 
and Wooley~\cite{Wool2} (for $d=3$)  (see also a more general form due to Wooley~\cite{Wool5}),   
 one can derive a much stronger  bound $N^{s(d)(1-2\alpha) + o(1)}$ when $1/2<\alpha<1$.
 
In fact, a  special case of the {\it Menshov--Rademacher theorem\/}, see~\cite[p.~251]{KaSa},
  implies that      for almost all $\vx\in \Tor$ (with respect to the Lebesgue measure) we have  
 \begin{equation}
\label{eq:M-R}
  |S_d(\vx; N)| \leqslant  N^{1/2}(\log N)^{3/2+o(1)}, \qquad \text{as} \ N \to \infty .
\end{equation} 
For completeness we give a proof of~\eqref{eq:M-R}  in Appendix~\ref{app:A}.

Hence if  for $0<\alpha<1$  we  define the  set 
$$
\cE_{\alpha, d}=\{\vx\in \Tor:~|S_d(\vx; N)|\geqslant N^{\alpha} \text{ for infinitely many } N\in \N\},
$$ 
and  define
$$
\vartheta_d =\inf \{\alpha>0:~\lambda(\cE_{\alpha, d})=0\}, 
$$
 where we use  $\lambda(\cA)$ to denote the  {\it Lebesgue measure\/} of $\cA \subseteq \Tor$, 
then  by~\eqref{eq:M-R}   we have 
 $$
\vartheta_d \leqslant 1/2.
$$
In fact we make:

\begin{conj}\label{conj:1/2}  
For each integer $d\geqslant 2$ we have  
$$
\vartheta_d =1/2. 
$$
\end{conj}

Here we are mostly interested in the structure of the set of exceptional $\vx\in \Tor$ 
for which~\eqref{eq:M-R} does not hold. For convenience  we
call $\cE_{\alpha, d}$ the exceptional set for each $0<\alpha<1$ and $d\in \N$. Thus we study the exceptional sets 
$\cE_{\alpha, d}$ and show that they are massive enough   in a sense of {\it Baire categories\/}
and the {\it Hausdorff dimension\/}.

\subsection{Main results}

Recall that a subset of $\R^{d}$ is called {\it nowhere dense \/} if its closure in $\R^{d}$ has an empty interior.  We now recall the following:

\begin{definition}
A subset of $\R^{d}$ is of the {\it first  Baire category\/} if it is a countable union of nowhere dense sets; otherwise it is called of the {\it second Baire category\/}. 
\end{definition}

For the basic properties and various applications of Baire categories we refer to~\cite{Oxtoby, SteinShakarchi}.

We now show that the complements of the sets  $\cE_{\alpha, d}$ are  small.

\begin{theorem}\label{thm:Baire}
For each $0<\alpha<1$ and integer $d\geqslant 2$, the subset $\Tor \setminus \cE_{\alpha, d}$ is of the  first Baire category.
\end{theorem}

Alternatively,  Theorem~\ref{thm:Baire} is equivalent to the statement that the complement 
$\Tor \setminus\Xi_{d}$ to the set 
\begin{align*}
\Xi_{d}= \bigl\{\vx\in \Tor:~\forall \varepsilon>0, \ |S_d(\vx; N)| &\geqslant N^{1-\varepsilon}\\
& \text{ for infinitely many } N\in \N \bigr\}
\end{align*}
is of first category. Indeed, let $\alpha_j =1-1/j$, $j=1, 2, \ldots$. Then   
$$
\Tor \setminus\Xi_{d}= \Tor \setminus\(\bigcap_{j=1}^{\infty}\cE_{\alpha_j, d}\) = 
 \bigcup_{j=1}^{\infty} \( \Tor \setminus \cE_{\alpha_j, d}\)
$$ 
is a countable union of  first category sets, and is of  first category too.
Since  also  for any $0<\alpha<1 $ we have   $\Xi_{d} \subseteq \cE_{\alpha, d}$, we obtain the desired equivalence.

For sets of  Lebesgue measure zero, it is common to use  the {\it Hausdorff dimension\/} to describe their size; for the properties of the Hausdorff dimension and its applications we refer to~\cite{Falconer, Mattila1995}.  We recall that for $\cU \subseteq \R^{d}$
$$
\diam \cU = \sup\{\| u-v\|_{L^2}:~u,v \in \cU\}
$$
where $\|w\|_{L^2}$ is the Euclidean norm in $\R^{d}$. 

\begin{definition} 
\label{def:Hausdorff}
The  Hausdorff dimension of a set $\cA\subseteq \R^{d}$ is defined as 
\begin{align*}
\dim \cA=\inf\Bigl\{s>0:~\forall \, & \eps>0,~\exists \, \{ \cU_i \}_{i=1}^{\infty}, \ \cU_i \subseteq \R^{d},\\
&  \text{such that } \cA\subseteq \bigcup_{i=1}^{\infty} \cU_i \text{ and } \sum_{i=1}^{\infty}\(\diam\cU_i\)^{s}<\eps \Bigr\}.
\end{align*}
\end{definition}

We show that  for $d\geqslant 2$ and any $0<\alpha<1$ the exceptional set  $\cE_{\alpha, d}$ 
is {\it everywhere rich\/} in  a sense that its intersection 
$$
 \cE_{\alpha, d}\(\fQ\)  =  \cE_{\alpha, d} \cap \fQ,
$$
with any  cube $\fQ \subseteq \Tor$, is of  positive Hausdorff dimension, and give an explicit 
lower bound on this dimension.

We now define 
\begin{equation}  
\label{eq:kappa}
\kappa_d =  \max_{\nu =1, \ldots, d} \min\left\{ \frac{1}{2\nu} ,  \frac{1}{2d-\nu} \right \}.
\end{equation}
We note that
$$
\lim_{d \to \infty} d \kappa_d =\frac{3}{4}
$$
and in fact  if $3 \mid d$ then $\kappa_d=3/4d$.

\begin{theorem}\label{thm:dimension}
For  each $0<\alpha<1$ and any cube $\fQ \subseteq \Tor$   we have 
\begin{itemize}
\item[(i)] for $d=2$, 
$$
\dim \cE_{\alpha, 2} \(\fQ\) \geqslant \min\{ 3/2,\, 3(1-\alpha)\}; 
$$
\item[(ii)] for $d\geqslant 3$, 
$$
\dim \cE_{\alpha, d} \(\fQ\) \geqslant  \min\left\{ \kappa_d,  2 \kappa_d (1-\alpha) \right\}.
$$
\end{itemize}
\end{theorem}

Note that for $0<\alpha \leqslant 1/2$ Theorem~\ref{thm:dimension} asserts that  
$$
\dim \cE_{\alpha, d} \(\fQ\)  \geqslant 
\begin{cases}
 3/2, & \text{for} \ d = 2,\\
\kappa_d  & \text{for} \ d \geqslant 3.
\end{cases}
$$
However Conjecture~\ref{conj:1/2} asserts that for 
any $\alpha\in (0, 1/2)$ and any integer $d\geqslant 2$  we have  $\lambda(\cE_{\alpha, d})>0$ and hence we expect 
$$
\dim \cE_{\alpha, d}=d, 
$$
and even stronger 
$$
\dim  \cE_{\alpha, d} \(\fQ\)=d, 
$$
for any cube $\fQ \subseteq \Tor$.
We remark that in fact we expect  $\lambda(\cE_{\alpha, d}) =1$ for any  $\alpha\in (0, 1/2)$, see Conjecture~\ref{conj:lambda_Omega} below.

Our approach to  Theorem~\ref{thm:dimension}  is based on a version  of the classical   {\it Jarn\'ik--Besicovitch theorem\/}, see~\cite[Theorem~10.3]{Falconer} or~\cite{BDV} and on the investigation  of the distribution of large values of 
rational exponential sums with
prime denominators. This question  is of independent interest and it also gives us an opportunity to mention very interesting but perhaps not so well-known 
results of Knizhnerman and  Sokolinskii~\cite{KnSok1, KnSok2} about large and small values of rational exponential sums. 

Furthermore, we also investigate the monomial sums 
 $$
\sigma_d(x; N)=\sum_{n=1}^{N} \e(xn^d)
$$
to which the above technique does not apply. 
Similarly to the sets $\cE_{\alpha, d}$, for each  $0<\alpha<1$ let 
$$
\sE_{\alpha, d}= \{x\in [0, 1):~|\sigma_d(x; N)| \geqslant N^{\alpha}\text{ for infinitely many } N\in \N\}.
$$

Similarly to Theorem~\ref{thm:Baire} and Theorem~\ref{thm:dimension},  we  also obtain the corresponding results for the monomial sums.

\begin{theorem}
\label{thm:Baire-mon}
For each $0<\alpha<1$ and each integer  $d\geqslant 2$,   the set  $[0, 1)\setminus \sE_{ \alpha, d}$ is of first Baire category.
\end{theorem}


We also show the positivity of the Hausdorff dimension of 
$$
\sE_{\alpha, d}(\fI)=\sE_{\alpha, d}\cap \fI
$$ 
for any interval $\fI\subseteq \T$.   In analogy to Theorem~\ref{thm:dimension} we have the following result.

\begin{theorem}
\label{thm:dimension-mon}
For  each $0<\alpha<1$ and any interval $\fI\subseteq \T$,   we have 
\begin{itemize}
\item[(i)] for $d=2$, 
$$
\dim \sE_{ \alpha, 2}(\fI)\geqslant \min \{1, 2(1-\alpha)\}; 
$$
\item[(ii)] for $d\geqslant 3$, 
$$
\dim \sE_{ \alpha, d}(\fI)\geqslant  (1+1/d) \min\left \{2/(d +2),  1-\alpha \right \}.
$$
\end{itemize} 
\end{theorem}
 
Note that for $0<\alpha \leqslant 1/2$  Theorem~\ref{thm:dimension-mon}~(i), for the case $\fI=\T$, asserts that  
$$
\dim \sE_{\alpha, 2}=1.
$$
In fact only the case of $\alpha = 1/2$ is of interest as for $\alpha < 1/2$, and this is instant 
from the result of  Fiedler, Jurkat and K\"orner~\cite[Theorem~2]{FJK}. 

For $0<\alpha\leqslant d/(d+2)$ with $d\geqslant 3$  Theorem~\ref{thm:dimension-mon}~(ii), for the case $\fI=\T$, asserts that 
$$
\dim \sE_{\alpha, d}\geqslant \frac{2(d+1)}{d(d+2)}.
$$
However we conjecture that for each $0<\alpha\leqslant 1/2$ and each $d\geqslant 2$ one  has  
$$
\dim \sE_{\alpha, d}=1
$$ 
and perhaps even stronger 
$$
\dim \sE_{\alpha, d}\(\fI\)=1, 
$$
for any interval $\fI \subseteq \T$. 

\subsection{Applications to uniform distribution modulo one}
\label{sec:ud mod1}

 A quantitative  way to describe the {\it uniformity of distribution modulo one} is given by the  {\it discrepancy\/}, see~\cite{DrTi}. 

\begin{definition}
\label{def:Discr}
Let $x_n$, $n\in \N$,  be a sequence in $[0,1)$. The {\it discrepancy\/}  of this sequence at length $N$ is defined as 
$$
D_N = \sup_{0\le a<b\le 1} \left |  \#\{1\le n\le N:~x_n\in (a, b)\} -(b-a) N \right |.
$$
\end{definition}

 Recalling that a sequence is uniform distributed modulo one if and only if the corresponding discrepancy 
$$
D_N=o(N) \qquad \text{as} \ N \to \infty,
$$
see~\cite[Theorem~1.6]{DrTi} for a proof. We note that sometimes in the literature the scaled quantity $N^{-1}D_N $ is   called the discrepancy, 
since our argument looks cleaner with Definition~\ref{def:Discr}, we adopt it here. 
 
For $\vx\in \Tor$ and the sequence 
$$
x_1n+\ldots +x_d n^{d}, \qquad n\in \N,
$$
we  denote by $D_d(\vx; N)$ the corresponding discrepancy.  Motivated by the work of Wooley~\cite[Theorem~1.4]{Wool3}, 
the authors~\cite{ChSh2}  
have shown that for almost all $\vx\in \Tor$ with $d\ge 2$ one has 
\begin{equation}
\label{eq:D1/2}
D_d(\vx; N)\le N^{1/2+o(1)} \qquad \text{as} \ N \to \infty.
\end{equation}
In view of Lemmas~\ref{lem:KS} and~\ref{lem:KH} below,
Conjecture~\ref{conj:1/2} is equivalent to the statement  that the exponent $1/2$ in~\eqref{eq:D1/2} cannot be improved.

Thus,  the bound~\eqref{eq:D1/2},  combined with Lemma~\ref{lem:KH} below, provides yet another way to obtain that 
$$
S_d(\vx; N)\ll N^{1/2+o(1)},\qquad \text{as} \ N \to \infty,
$$
holds for almost all $\vx\in \Tor$ (which is a slightly less precise version of~\eqref{eq:M-R}). 

Let
$$
\cD_{\alpha, d}=\{\vx\in \Tor:~D_d(\vx; N)\ge N^{\alpha} \text{ for infinitely many } N\in \N\}.
$$

\begin{theorem}
\label{thm:BaireD}
For each $0<\alpha<1$ and integer $d\ge 2$ the subset $\Tor\setminus \cD_{\alpha, d}$ is of the first Baire category. 
\end{theorem}

Note that  this is equivalent to the statement that the complement 
$\Tor \setminus\mathfrak{D}_{d}$ to the set 
\begin{align*}
\mathfrak{D}_{d}= \bigl\{\vx\in \Tor:~\forall \varepsilon>0,\  D_d(\vx; N) &\geqslant N^{1-\varepsilon}\\
& \text{ for infinitely many } N\in \N \bigr\}
\end{align*}
is of first Baire category. 


For any cube $\fQ\subseteq \T_d$ denote  $\cD_{\alpha, d}(\fQ)=\cD_{\alpha, d} \cap \fQ$.

\begin{theorem}
\label{thm:dimD}
For  each $0<\alpha<1$ and any $\fQ \subseteq \T_d$,  we have 
\begin{itemize}
\item[(i)] for $d=2$, 
$$
\dim \cD_{\alpha, 2}(\fQ)\geqslant \min\{ 3/2,\, 3(1-\alpha)\}; 
$$
\item[(ii)] for $d\geqslant 3$, 
$$
\dim \cD_{\alpha, d}(\fQ) \geqslant  \min\left\{ \kappa_d,  2 \kappa_d (1-\alpha) \right\}.
$$
\end{itemize}
\end{theorem}

In the case of 
monomials, 
For $x \in  [0,1)$ 
we  denote by $\Delta_d(x; N)$ the   discrepancy of  the sequence 
$xn^{d}$,  $n\in \N$ and set
$$
\sD_{\alpha, d}=\{x \in  [0,1):~\Delta_d(x; N)\ge N^{\alpha} \text{ for infinitely many } N\in \N\}.
$$

We have the following analogues of Theorems~\ref{thm:BaireD} and~\ref{thm:dimD}

\begin{theorem}
\label{thm:BaireD-mon}
For each $0<\alpha<1$ and integer $d\ge 2$ the subset $[0,1)\setminus \sD_{\alpha, d}$ is of the first Baire category. 
\end{theorem}

Furthermore, we also have the following result.  For an interval $\fI \subseteq \T$ denote $\sD_{\alpha, d}(\fI)=\sD_{\alpha, d} \cap \fI$.

\begin{theorem}
\label{thm:dimD-mon}
For  each $0<\alpha<1$ and any interval $\fI \subseteq \T$,  we have 
\begin{itemize}
\item[(i)] for $d=2$, 
$$
\dim \sD_{ \alpha, 2}(\fI)\geqslant \min \{1, 2(1-\alpha)\}; 
$$
\item[(ii)] for $d\geqslant 3$, 
$$
\dim \sD_{ \alpha, d}(\fI)\geqslant  (1+1/d) \min\left \{2/(d +2),  1-\alpha \right \}
$$
\end{itemize}
\end{theorem}

%
%

We remark that the case $d=1$ is a special case. For the  linear sequence $(nx)$ the celebrated result of  Khintchine, see~\cite[Theorem~1.72]{DrTi}, implies that  for almost all $x\in [0,1)$ one has 
$$
\cD_1(x; N)\le N^{o(1)},\qquad \text{as} \ N \to \infty.
$$

 \section{Preliminaries} 
 
 \subsection{Notation and conventions}

Throughout the paper, the notations $U = O(V )$, 
$U \ll V$ and $ V\gg U$  are equivalent to $|U|\leqslant c|V| $ for some positive constant $c$, 
which throughout the paper may depend on the degree $d$ and occasionally on the small real positive 
parameters $\varepsilon$ and $\delta$.  
 
 We use $\# \cX$ to denote the cardinality of set $\cX$.  
 
 The letter $p$, with or without a subscript,  always denotes a prime number. 
 
 We always identify $\Tor$ with half-open unit cube $[0, 1)^d$, in particular we
 naturally associate Euclidean norm  $\|x\|_{L^2}$ with points $x \in \Tor$.

We say that some property holds for almost all $\vx \in \Tor$ if it holds for a set 
 $\cX \subseteq \Tor$ of  Lebesgue measure  $\lambda(\cX) = 1$. 
 
We always keep the subscript $d$ in notations for our main objects of interest such as 
$\cE_{\alpha, d}$, $S_d(\vx; N)$ and $\Tor$, but sometimes suppress
it in auxiliary quantities.

\subsection{Complete rational exponential sums and uniform distribution}

We first  recall the classical {\it Weil bound\/}, see, for example,~\cite[Chapter~6, Theorem~3]{Li}. 
For a prime $p$,  let $\F_p$ denote the finite field of $p$ elements, which we identify with the set
$\{0, \ldots, p-1\}$, and  let $\F_p^*=\F_p\setminus \{0\}$. Furthermore let $\ep(z)=\e(z/p)$.  

\begin{lemma}\label{lem:Weil}
Let $f\in \F_p[X]$ be a nonconstant polynomial  of  degree $ \deg f \leqslant d$.  Then we have  
$$
 \sum_{\lambda \in \F_p}  \ep\(f(\lambda)\)   \ll \sqrt{p}.
$$
\end{lemma}

Applying Lemma~\ref{lem:Weil} and the completion  technique (see~\cite[Section~12.2]{IwKow}) fwe derive the following bounds for incomplete sums. For $1\le N\le p$ one has   
\begin{equation}
\label{eq:com}
\sum_{n=1}^N \e_p(f(\lambda)) \ll \sqrt{p} \log p.
\end{equation}

Next, we consider discrete cubic boxes 
\begin{equation}
\label{eq:box}
\fB=\cI_1\times \ldots \times \cI_d \subseteq \F_p^{d}
\end{equation}
with the side length
$$
\ell(\fB) = L, 
$$
where  $\cI_j = \{k_j+1, \ldots, k_j +L\}$ is a set of $L \leqslant p $ consecutive integers,
(reduced modulo $p$ if $k_j +L \ge p$), $j =1, \ldots, d$.

We formulate the following easy consequence   of  the {\it Koksma--Sz\"usz inequality\/}, see~\cite[Theorem~1.21]{DrTi}.

\begin{lemma}
\label{lem:KS}
Let $\bxi_i \in \F_p^{d}$,  $1\leqslant i\leqslant I$, be a sequence of $I$ vectors over $\F_p$ and  let $\fB\subseteq \F_p^{d}$ be a box. Let 
$$
R=\#\{\bxi_i \in \fB:~1\leqslant i\leqslant I\}.
$$ 
Then we have 
$$
\left| R-\# \fB I p^{-d} \right| \ll (\log p)^{d}\max_{\vh \in \F_p^{d} \setminus \{ \0\}}\left|  \sum_{i=1}^{I} \ep(\left\langle\bxi_i,  \vh \right\rangle) \right|, 
$$
where $\left\langle\bxi,  \vh \right\rangle$ denotes the scalar product of two vectors $\bxi, \vh \in \F_p^{d}$.
\end{lemma}

\subsection{Distribution of large rational exponential sums}

 For  a vector $\va = (a_1, \ldots, a_d)\in \F_p^d$ we consider the rational exponential  sum
$$
T_{d, p} (\va) =S_d(\va/p;p)=\sum_{n=1}^{p}\ep\(a_1 n+\ldots +a_d n^{d} \). 
$$

We need some results about the density of the vectors $\va \in \F_p^d$ for which the sums $T_{d, p} (\va)$ are large.

For $d=2$ the answer  to the question is trivial due to the following property of Gaussian sums, 
see~\cite[Equation~(1.55)]{IwKow}.

\begin{lemma}
\label{lem:Gauss}
Let $p\ge 3$ and $a, b\in \F_p$ with $b\neq 0$, then 
$$
\left|\sum_{n=0}^{p-1}\ep\left(an+bn^{2}\right)\right|=\sqrt{p}.
$$
\end{lemma}

We now investigate the case of $d\geqslant 3$.   
For this, we  define 
$$
\omega_d = \liminf_{p\to \infty}\frac{1}{\sqrt{p}}
\, \max_{\substack{\va = (a_1, \ldots, a_d)\in \F_p^d\\ a_d \ne 0}}\,  \left| T_{d, p} (\va)\right|.
$$  
From the classical method of Mordell~\cite{Mord} we have  
\begin{equation}  
\label{eq:Mord}
\sum_ {\va \in \F_p^d} \left| T_{d, p} (\va)\right|^{2d} = d! p^{2d} + O(p^{2d-1}).
\end{equation}

Hence, taking into account the contribution $\left| T_{d, p} (\0)\right|^{2d} = p^{2d}$ from the zero vector $\va \ne \0$
and estimating the contribution from $O\(p^{d-1}\)$ vectors with $a_d=0$ by Lemma~\ref{lem:Weil}, we obtain
$$
\sum_{\substack{\va = (a_1, \ldots, a_d)\in \F_p^d\\ a_d \ne 0}}\left| T_{d, p} (\va)\right|^{2d} = \(d!-1\) p^{2d} + O(p^{2d-1}), 
$$
which trivially implies that 
$$
\omega_d \geqslant  \(d!-1\) ^{1/2d}. 
$$
Knizhnerman and  Sokolinskii~\cite{KnSok1,KnSok2} have given stronger lower bounds, asymptotically for $d\to \infty$
and also for small values of $d$,  for example, $\omega_3 \geqslant \sqrt{3}$. 

Furthermore, by~\cite[Theorem~1]{KnSok1} we have 

\begin{lemma}
\label{lem:LB Dens}
For every integer $d \geqslant 2$ there are some positive constants $c_d$ and $\gamma_d$ 
such that 
$$
\left| T_{d, p} (\va)\right|  \geqslant \gamma_d \sqrt{p} 
$$
for a set $\cL_p \subseteq  \F_p^d$ of cardinality  $\#\cL_p \geqslant c_dp^d$. 
\end{lemma}

We now show that the vectors $\va \in \F_p^d$ for which the sums $T_{d, p} (\va)$ reach their extreme values are 
reasonably densely distributed. 
That is. we intend to show that the set $\cL_p$ of Lemma~\ref{lem:LB Dens}  is quite dense. Before this we provide 
a  result on the distribution of monomial curves.

\begin{lemma}   
\label{lem:k}
Let $(a_1, \ldots, a_k)\in (\F_p^*)^{k}$, $k\geqslant 2$.  Then  there exists a positive constant $C$ which depends only 
on $k$ such that for any box $\fB$ as in~\eqref{eq:box}  with the side length $L\geqslant C p^{1-1/2k} \log p$  we have 
$$
 \#\left\{\lambda\in \F_p^{*}:~(a_1 \lambda , \ldots, a_k \lambda^{k})\in \fB \right\} \geqslant 0.5  L^{k}p^{1-k}. 
$$   
\end{lemma}

\begin{proof}
For a nonzero vector $\vh = (h_1, \ldots, h_k) \in \F_p^{k} \setminus \{\0\}$  the Weil bound, see Lemma~\ref{lem:Weil}, gives 
$$
\sum_{\lambda\in \F_p^{*}} \ep\left(  \sum_{j=1}^{k} \lambda^{j}a_jh_j\right )\ll p^{1/2}.
$$ 
Combining this bound with Lemma~\ref{lem:KS}, we finish the proof.
\end{proof}

Clearly we can replace a lower bound $0.5  L^{k}p^{1-k}$ of Lemma~\ref{lem:k} 
with an asymptotic formula $(1+o(1))  L^{k}p^{1-k}$ for slightly larger values of $L$, namely, 
if $L^{-1} p^{1-1/2k} \log p\to 0$ as $p\to \infty$. We also note that  Lemma~\ref{lem:k} still holds for the case $k=1$.


\begin{lemma}\label{lem:box dense} Fix $d\geqslant 3$. There is an  constant $C > 0$ depending only on $d$, such that for 
a box
$\fB \subseteq \F_p^{d}$ as in~\eqref{eq:box} with the side length $L \geqslant C p^{1-\kappa_d} \log p$, where $\kappa_d$ is  as in~ \eqref{eq:kappa},  and $\cL_p$ as in Lemma~\ref{lem:LB Dens}, there 
is $\va \in \fB\cap \cL_p$.
 \end{lemma}

\begin{proof}
Adjusting $C$ if necessary, we can assume that $p$ is large enough. 

Clearly, if $(a_1, \ldots, a_d) \in \cL_p$ then for any $\lambda\in \F_p^*$ 
we also have $(a_1\lambda, \ldots, a_d \lambda^d) \in \cL_p$.  Let $k$ be an integer such that  
$$
\kappa_d=\min\{ 1/2k ,  1/(2d-k) \}.
$$
By Lemma~\ref{lem:LB Dens} we conclude that there exists $(a_1, \ldots, a_k)\in \F_p^{k}$ with $a_i\neq 0$ for each $1\leqslant i\leqslant k$ such that 
$$
\# \cL_p \cap \left(\left\{a_1, \ldots, a_k\right\}\times \F_p^{d-k}\right) \gg p^{d-k}.
$$
For convenience we denote this set  by $\cL_{p, k}^{*}$.

Let $\fB= \F_p^{d}$ be a box with the side length $\ell(\fB) = L$, which we decompose in a natural way 
as $\fB=\fB_1\times \fB_2\subseteq \F_p^{k}\times \F_p^{d-k}$

Note that we have  $\#\fB_1=L^{k}$. Let 
$$
\Lambda_k=\{\lambda\in \F_p^{*}:~(\lambda a_1, \ldots, \lambda^{k}a_k)\in \fB_1\}.
$$
Then Lemma~\ref{lem:k} implies that 
\begin{equation}
\label{eq:Lk}
\# \Lambda_k \geqslant 0.5 L^{k}p^{1-k}
\end{equation}
provided the condition
\begin{equation}
\label{eq:L 1/2k}
L\geqslant C p^{1-1/2k} \log p
\end{equation}
is satisfied with a sufficiently large $C$.

We now fix a vector  $\vh=(h_{k+1}, \ldots, h_d)\in \F_p^{d-k}\setminus \{\0\}$ and consider the double exponential sums
$$
W(\vh)=\sum_{(a_1, \ldots, a_d)\in \cL_{p, k}^{*}}
\sum_{\lambda\in \Lambda_k} \ep \left(\sum_{j=k+1}^{d}h_ja_j \lambda^{j}\right).
$$
By the Cauchy-Schwarz inequality 
\begin{align*}
|W(\vh) |^2 & \leqslant \# \cL_{p, k}^*   \sum_{(a_1, a_2, \ldots, a_d)\in \cL_{p, k}^*} \left|\sum_{\lambda\in \Lambda_k}
\ep\(\sum_{j=k+1}^d h_ja_j \lambda^j\)\right|^2\\
& \leqslant \# \cL_p^*   \sum_{(a_k, \ldots, a_d)\in \F_p^{d-k}} \left|\sum_{\lambda\in \Lambda}
\ep\(\sum_{j=k+1}^d h_ja_j \lambda^j\)\right|^2.
\end{align*}
Now using that for any $z\in \C$ we have $|z|^2 = z \overline z$, and then changing the order of summations,  we 
obtain
\begin{align*}
|W(\vh) |^2 & \leqslant \# \cL_{p,k}^*   \sum_{\lambda,\mu \in \Lambda_k}
\sum_{(a_k, \ldots, a_d)\in \F_p^{d-k}}  
\ep\(\sum_{j=k+1}^d h_ja_j \(\lambda^j-\mu^j\)\) \\
 & \leqslant \# \cL_{p, k}^*   \sum_{\lambda,\mu \in \Lambda_k} \prod_{j=k+1}^d
\sum_{a_j \in \F_p}  
\ep\(\  h_ja_j \(\lambda^j-\mu^j\)\) .
\end{align*}
By the orthogonality of exponential functions, the last sum  vanishes unless for every $j=k+1, \ldots, d$ 
we have  $h_j \(\lambda^j-\mu^j\) = 0$. Since  $\vh$ is a nonzero vector of $ \F_p^{d-k}$, this is possible for at most 
$2d\#\Lambda_k$ pairs $\(\lambda,\mu\) \in \Lambda_k^2$, and in the case the inner sum is equal to $p^{d-k}$. 
Hence, for any nonzero vector $\vh \in \F_p^{d-k}$ we have 
$$
|W(\vh)|^ 2\ll  \# \cL_{p, k}^*   \#\Lambda_k p^{d-k}.
$$
Using that $\#\cL_{p, k}^{*}\gg p^{d-k}$, we now obtain
\begin{equation}\label{eq:Wh}
|W(\vh)|\ll  \# \cL_{p, k}^*   (\#\Lambda_k)^{1/2}.
\end{equation}

Let $R$ be the number of the  vectors $\(a_{k+1},\ldots a_d, \lambda\)\in \cL_{p, k}^{*}\times \Lambda_k$ such that 
\begin{equation}
\label{eq:B2}
(\lambda^{k+1}a_{k+1}, \ldots, \lambda^{d} a_d) \in \fB_2.
\end{equation}
Combining the bound~\eqref{eq:Wh} with Lemma~\ref{lem:KS},  we obtain
$$
R= \# \cL_{p, k}^{*} \# \Lambda_k ( L/p)^{d-k}+ O(\# \cL_{p, k}^{*} (\# \Lambda_k)^{1/2} (\log p)^{d-k}).
$$
Thus we conclude that $R>0$ when 
$$
L^{d-k} \# (\Lambda_k)^{1/2}\geqslant C_0p^{d-k} (\log p)^{d-k}
$$
for some constant $C_0$ depending only on $d$ and $k$. By~\eqref{eq:Lk} this condition becomes  
$$
L^{d-k}(0.5L^{k}p^{1-k})^{1/2}\geqslant C_0 p^{d-k} (\log p)^{d-k},
$$ 
and hence it is enough to request that 
\begin{equation}
\label{eq:L 1/(2d-k)}
L \geqslant  C p^{1-1/(2d-k)}(\log p)^{(d-k)/(d-k/2)}
\end{equation}
for a sufficiently large  constant $C$.

Combining the conditions~\eqref{eq:L 1/2k} and~\eqref{eq:L 1/(2d-k)}, and recalling  the definition of $\kappa_d$ in~\eqref{eq:kappa}, 
we conclude that there exists a large enough constant  $C$  such that the inequality  
$$L\geqslant C p^{1-\kappa_d}\log p$$ is sufficient  to guarantee that for some 
$\(a_{k+1},\ldots a_d, \lambda\)\in \cL_{p, k}^{*}\times \Lambda_k$  we have~\eqref{eq:B2}. 
Since we always have $(a_1\lambda, \ldots, a_k \lambda^{k})\in \fB_1$ when $\lambda \in  \Lambda_k$ and so the result now follows.  
\end{proof}

\begin{cor}\label{cor:dense}
Let $\cL_{p}$ be  defined as in Lemma~\ref{lem:LB Dens}. Then for any $k\in \N$ the set 
$$
 \bigcup_{\substack {p \geqslant k \\ p\text{ is prime} } } \cL_{p} \subseteq \Tor
$$ 
 is  dense in $\Tor$.   
 \end{cor}

\begin{proof}
Let $\sfB$ be a box of $\Tor$ with the side length
$$
\ell(\sfB) = 2 C p^{-\kappa_d}\log p, 
$$  
where  $C$ is as in  Lemma~\ref{lem:box dense}. Define 
$$
\fB= \left \{ \va\in \F_p^{d}:~  \va/p \in \sfB\right \}.
$$
By Lemma~\ref{lem:k} there exists $\vb\in \fB$  such that 
$$
|T_{d, p}(\vb)| \geqslant \gamma_d  \sqrt{p}
$$
provided that  $p$ is large enough. Thus, we conclude that 
$$
 \vb/p \in \cL_p \cap \sfB. 
$$
Since this holds for any box $\sfB$ of $\Tor$,   the result follows.
\end{proof}

\begin{remark} 
\label{rem:why1}
{\rm For the case  $d=2$, Corollary~\ref{cor:dense}  follows immediately from Lemma~\ref{lem:Gauss}. 
However in general Lemma~\ref{lem:Gauss} does not hold for $d\geqslant 3$ and in fact  
$\va \in \F_p^{d}$ with vanishing sums $T_{d, p}(\va)=0$ are often densely distributed as well.

For instance, for $d\geqslant 3$ and a  prime number $p$ with $\gcd(d, p-1)=1$,  the map: $x\rightarrow x^{d}$ permutes $\F_p$. 
Hence, for any $\lambda \in \F_p^*$ we have 
\begin{align*}
\sum_{n=0}^{p-1}\ep\(\sum_{j=1}^d \binom{d}{j} \lambda^j   n^j\) &
=  \sum_{n=0}^{p-1}\ep\((\lambda n+1)^{d}-1\)\\
& = \sum_{n=0}^{p-1}\ep\(n^{d}-1\)=  \sum_{n=0}^{p-1}\ep\(n\)=0.
\end{align*}
Assuming $p > d$ we see that 
$$
 \binom{d}{j}  \not \equiv 0 \pmod p, \qquad j =1, \ldots, d.
$$
By Lemma~\ref{lem:k} for any   box $\fB  \subseteq \F_p^{d}$ with the side length $\ell(\fB) \geqslant C p^{1-1/2d}\log p$
for some constant $C$ there exists $\lambda\in \F_p^{*}$ such that 
$$
\(  \binom{d}{1} \lambda , \ldots,\binom{d}{d}  \lambda^{d} \)\in \fB.
$$ 
Therefore we conclude that for any $k\in \N$ the set
$$
\bigcup_{\substack {p \geqslant k \\ p\text{ prime} } }\{\va/p:~\va \in \F_p^{d}, \ T_{d, p}(\va)=0\}
$$
is a dense subset of $\Tor$.
}
\end{remark}

\subsection{Large Weyl sums}
\label{sec:weyl sum}

We are going to show that the small neighbourhood of $\cL_{p}$  still have large exponential sums. 
Namely let $\sfB(\vx, \delta)$ denotes the cubic box centered at $\vx\in \Tor$ with  the side length  
$$
\ell\(\sfB(\vx, \delta)\) = 2\delta>0.
$$

 For each   $\tau> 0$ and a prime $p$ we define 
$$
\L_{\tau, p} =\bigcup_{\va\in \cL_{p}} \sfB(\va/p, p^{-\tau}). 
$$

We also use   $\gamma_d$  from  Lemma~\ref{lem:LB Dens}.

We use the following version of summation by parts.
Let $a_n$ be a sequence and for each $t\ge 1$ denote   
$$
A(t)=\sum_{1\le n\le t} a_n.
$$
Let $\psi: [1, N]\rightarrow \R$ be a differential function. Then 
$$
\sum_{n=1}^{N} a_n \psi(n)=A(N)\psi(N)-\int_{1}^{N} A(t)\psi'(t)dt.
$$

\begin{lemma} 
\label{lem: large sum N}
 Let $\vx\in \L_{\tau, p}$ for some  $\tau>0$ and prime $p$. There exists an absolute constant $c=c(d)$ such that if 
$$
c p^{\tau/d} (\log p)^{-1}\geqslant  N \geqslant  p \mand p \mid N
$$ 
then 
$$
|S_d(\vx; N)| \gg  N p^{-1/2}.
$$
\end{lemma}

\begin{proof} For any $\vx=(x_1, \ldots, x_d)\in \L_{\tau, p}$ there exist $\va=(a_1, \ldots, a_d) \in \cL_p$ such that 
$$
\left\| (x_1, \ldots, x_d)- (a_1/p, \cdots, a_d/p)  \right\|_{L^\infty} <p^{-\tau}, 
$$
where $\left\| \vz \right\|_{L^\infty} $ is the $L^\infty$-norm in $\R^d$.  Let $\delta_j=x_j-a_j/p, 1\le j\le d$. Applying summation by parts we obtain
\begin{equation}\label{eq:Abel}
\begin{split}
S_d(\vx; N) &-S_d(\va/p; N)\\
&=\sum_{n=1}^N\e_p \left (\sum_{j=1}^da_jn^j \right ) \left  (\e\(\sum_{j=1}^d \delta_j n^j\)-1 \right )\\
&=S_d(\va/p; N) \psi(N)-\int_{1}^{N} A(t)\psi'(t)dt,
\end{split}
\end{equation} 
where 
$$
A(t)=\sum_{1\le n\le t} \e_p \left (\sum_{j=1}^da_jn^j \right ) \mand
\psi(t)=  \e\(\sum_{j=1}^d \delta_j t^j\)-1 .
$$
Note that for any $u\in \R$ we have 
$$
|\e(u)-1|\le 2|u|.
$$
Combining with $|\delta_j|<p^{-\tau}, 1\le j\le d$ we obtain 
\begin{equation}
\label{eq:psi}
|\psi(N)|\le 2dN^{d}p^{-\tau}.
\end{equation}
For the integral part of~\eqref{eq:Abel} we derive 
\begin{equation*}
\begin{split}
\int_{1}^N A(t)  &\e\(\sum_{j=1}^d \delta_j t^j\) \(2 \pi i \sum_{j=1}^d j \delta_j t^{j-1}\)dt\\
&\le 2\pi p^{-\tau}\max_{1\le t\le N} |A(t)| \int_{1}^N \sum_{j=1}^d j t^{j-1} dt\\
&\le 4d \pi N^d p^{-\tau} \max_{1\le t\le N} |A(t)|.
\end{split}
\end{equation*}
Thus combining with~\eqref{eq:psi} and the definition of $A(t)$, and 
using bound~\eqref{eq:com} on incomplete sums,  we derive 
\begin{equation} 
\label{eq:DS}
\begin{split}
S_d(\vx; N) -S_d(\va/p; N) &\le  8d \pi N^d p^{-\tau}\max_{1\le t\le N} |A(t)|\\
& \le c_0 N^{d+1}  p^{-\tau-1/2} \log p,
\end{split}
\end{equation}
where $c_0>0$ is some  constant which  depends on $d$ only.

%

Since $p\mid N$,   using the periodicity of function $\ep(n)$,  we obtain 
\begin{equation}
\label{eq:Sa_Ta}
|S_d(\va/p; N)|  =  Np^{-1} |T_{d, p}(\va)|   \geqslant 0.5 \gamma_d N/p^{1/2}.
\end{equation}
Combining~\eqref{eq:DS} and~\eqref{eq:Sa_Ta} we obtain
$$
\left|S_d(\vx; N) \right| \geqslant    0.5 \gamma_d N p^{-1/2} -c_0 N^{d+1}  p^{-\tau-1/2} \log p  \ge 0.25  \gamma_d N p^{-1/2} 
$$
provided  
$$
N \le c   p^{\tau/d}(\log p)^{-1},
$$ 
for a sufficiently small constant $c$ (depending only on $d$), 
which gives the desired result.  
\end{proof}

We formulate some notation for our using on the lower bound of the Hausdorff dimension of $\cE_{\alpha, d}$.

\begin{lemma} 
\label{lem:twoboxes}
Let $\tau>d$.
For any $\varepsilon>0$ there exists $p_{\varepsilon, d}$ such that for any $p>p_{\varepsilon, d}$ and any cubic box $\sfB\subseteq \Tor$ with the side length  $\ell(\sfB) = p^{-\kappa_d+\varepsilon}$ there exists a box $\sfC\subseteq \sfB$ with the side length  $\ell(\sfC)  = p^{-\tau}$ and such that for 
$N=p \fl{c p^{\tau/d-1} (\log p)^{-1}}$, where  $c$ is  as in Lemma~\ref{lem: large sum N}, and all $\vx \in \sfC$, we have 
$$
|S_d(\vx; N)| \gg N^{1-d/2\tau}(\log N)^{-d/2\tau}. 
$$
\end{lemma}
\begin{proof}
Let $\sfB=\sfB(\vz, \ell(\sf B)/2)$ be the box.  For the box $\sfB(\vz, \ell(\sfB)/5)$,  
 Lemma~\ref{lem:box dense} implies that there exists a point   
$$\vc\in \cL_p \cap \sfB(\vz, \ell(\sfB)/5)$$
provided $p$ is large enough.  Let 
$
\sfC=\sfB(\vc, p^{-\tau}/2).
$
The condition $\tau>d$ gives $\tau>\kappa_d-\varepsilon$, and hence $\sfC\subseteq \sfB$.

By the choice of $N=p \fl{c p^{\tau/d-1} (\log p)^{-1}}$ and the condition $\tau >d$, 
Lemma~\ref{lem: large sum N} implies  that  for all $\vx\in \sfC$ we have 
$$
|S_d(\vx; N)| \gg  N^{1-d/2\tau}(\log N)^{-d/2\tau}
$$ 
which gives the desired result.
\end{proof}

\begin{definition}[$(a, b, c)$-patterns] \label{def:abc}  
{\rm Let $a>b>c>0$ and $a/b \in \Z$. Let $\sfB$ be a box with  with the side length $\ell(\sfB) = a$. We divide the box $\sfB$ into $(a/b)^{d}$ smaller  boxes in a natural way. For each of these $(a/b)^{d}$ boxes we pick a   smaller box, at  an  {\it arbitrary\/} location 
with the side length $c$. The resulting configuration  of  $(a/b)^{d}$ boxes with the side length $c$ is called an {\it $(a, b ,c)$-pattern\/}.} 
\end{definition}

An illustrative example of an $(a, b ,c)$-pattern is given  in Figure~\ref{fig1}.

\begin{figure}[H]
\centering 
\includegraphics[width=0.3\textwidth]{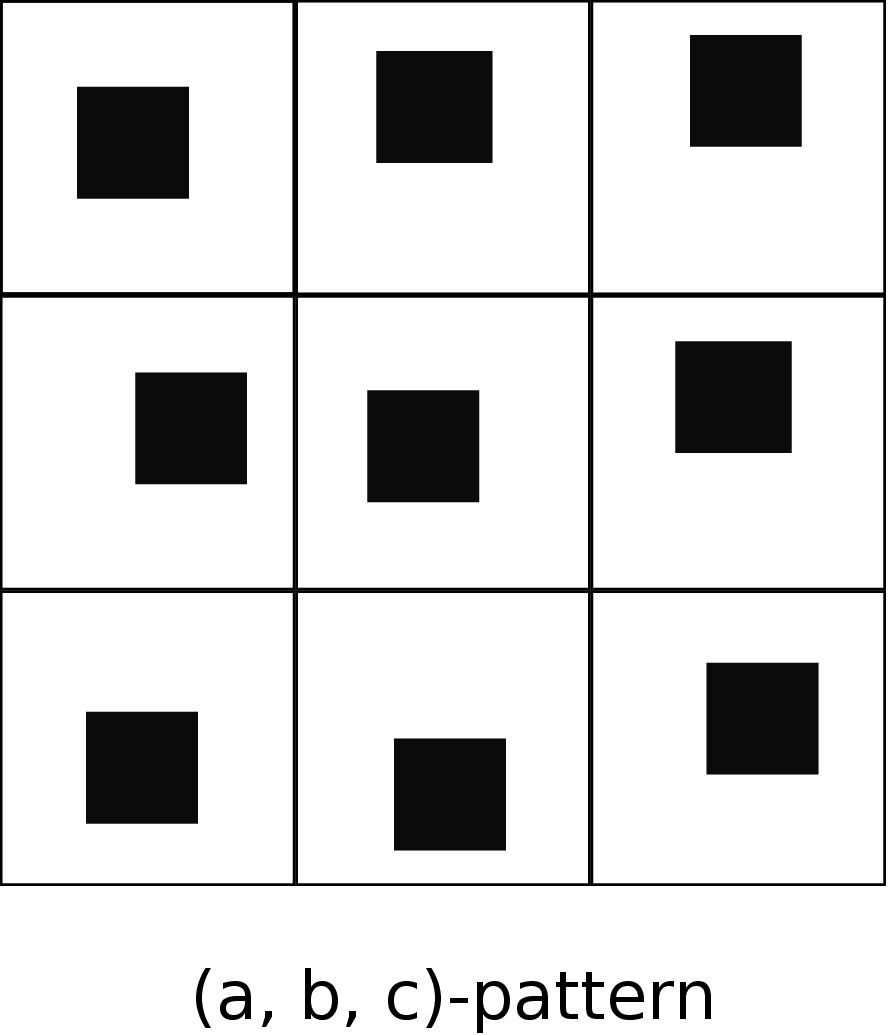}
\caption{An $(a, b, c)$-pattern with $a/b=3$   and $d=2$. }
\label{fig1}
\end{figure}

 We note that each $(a, b ,c)$-pattern is a subset of $\sfB$.
For our applications we find $(a, b, c)$-patterns such that the Weyl sums are large inside of the $(a/b)^{d}$  small boxes.
We show that  for any box $\sfB\subseteq \Tor$ there are  $(a, b, c)$-patterns which admit large Weyl sums. More precisely we have the following.

\begin{lemma}
\label{lem:manyboxes} 
Let $\tau, \varepsilon$ and $p_{\varepsilon,d}$ be the same as in Lemma~\ref{lem:twoboxes}. Let $p>p_{\varepsilon, d}$ and $\sfB\subseteq \Tor$ with the side length $\ell(\sfB) >10p^{-\kappa_d+\varepsilon}$. There exists $b$ such that $p^{-\kappa_d+\varepsilon}\leqslant b\leqslant 2p^{-\kappa_d+\varepsilon}$ and $\ell(\sfB)/b\in \Z$. Furthermore there exists a $(\ell(\sfB), b, p^{-\tau})$-pattern, which we denote by $\Upsilon_B$,  such that for 
$$
N=p \fl{c p^{\tau/d-1} (\log p)^{-1}}
$$ 
and all $\vx \in \Upsilon_{\sfB}$ we have 
$$
|S_d(\vx; N)|\gg N^{1-d/2\tau}(\log N)^{-d/2\tau}.
$$
\end{lemma}

\begin{proof}
Since $\ell(\sfB)/b\in \Z$, we divide the box $\sfB$ into $q=(\ell(\sfB)/b)^{d}$ smaller boxes of equal sizes in a natural way. We label them by $\sfB_1, \ldots, \sfB_q$ for convenience. 

For each $\sf B_i$, $1\leqslant i\leqslant q$,  Lemma~\ref{lem:twoboxes} asserts that there exists a box $\sf C_i\subseteq \sfB_i$ with the side length $p^{-\tau}$, and for all $\vx\in \sf C_i$ we have the desired bound. 

We finish the proof by taking  $\Upsilon_B=\bigcup_{i=1}^{q}\sf C_i$.
\end{proof}

\subsection{Hausdorff dimension of  a class of Cantor sets}

By a repeated  application of Lemma~\ref{lem:manyboxes}, we find  large Weyl sums on  a Cantor-like set.  This implies  a lower bound for the Hausdorff dimension  of $\cE_{\alpha, d}\(\fQ\)$. 
In this section we investigate a general construction of Cantor-like sets.

Now we  show the construction of the Cantor sets by iterating the construction of $(a, b, c)$-patterns.

Let 
$$
\bdelta=\(\delta_k\)_{k=1}^\infty \mand \bell=(\ell_k)_{k=1}^\infty 
$$ 
such that for each $k =1,2,  \ldots$, we have  
$$
\delta_k >\delta_{k+1} \mand \ell_k>\ell_{k+1}.
$$
 For convenience we also denote 
\begin{equation}
\label{eq:delta0}
\delta_0 = \lambda(\fQ)^{1/d}
\end{equation}
the side length of $\fQ$.
 
 For each $k\geqslant 0$ we ask that the  triple $(\delta_k, \ell_{k+1}, \delta_{k+1})$  satisfies the condition on $(a,b,c)$  in   Definition~\ref{def:abc}. In particular, we always assume that 
$$
\delta_{k}/\ell_{k+1}  \in \Z
$$ 
and  we denote   
\begin{equation}
\label{eq:qk}
q_{k+1}=\(\delta_{k}/\ell_{k+1}\)^{d}.
\end{equation}
for every $k =0, 1, \ldots$.

We start from the cube $\fQ$ and  take a $(\delta_0, \ell_1, \delta_1)$-pattern inside 
of $\fQ$.  

Let $\mathfrak{C}_1$ be the collection of these $q_1$ boxes. More precisely let 
$$
\fC_1=\{\sfB_i:~1\leqslant i\leqslant q_1\}.
$$ 
For each $\sfB_i$ we take a $(\delta_1, \ell_2, \delta_2)$-pattern  inside of $B_i$, and we denote these sub-boxes of $\sfB_i$ 
by $\sfB_{i, j}$ with $1\leqslant j\leqslant q_2$.  
Let 
$$
\fC_2=\{\sfB_{i, j}:~1\leqslant i\leqslant q_1, 1\leqslant j \leqslant q_2\}.
$$
Figure~\ref{fig2} shows an example of this construction.

\begin{figure}[H]
\centering 
\includegraphics[width=0.6\textwidth]{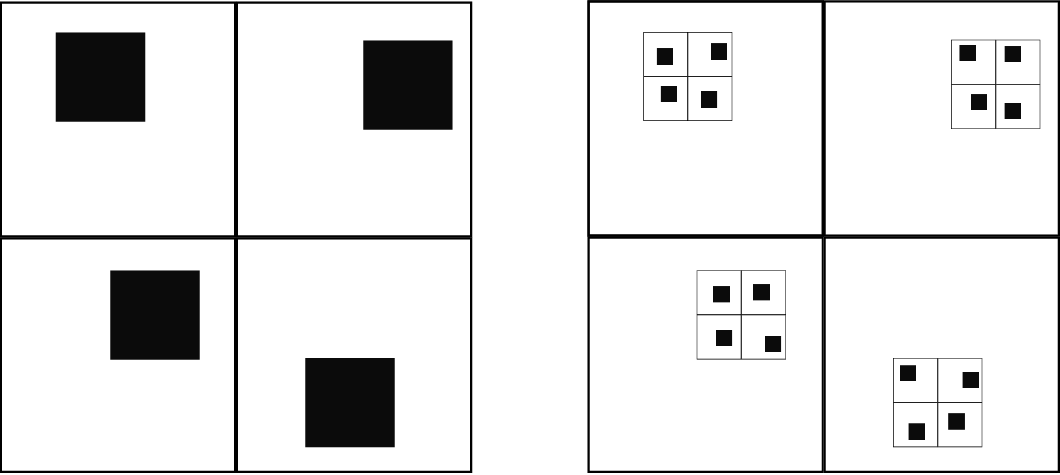}
\caption{The first two steps in the construction of $\FF$ with $\ell_1=1/2$ and $\delta_1/\ell_2=2$.}
\label{fig2}
\end{figure}

Suppose now we have $\fC_k$ which is a collection of  $\prod_{i=1}^{k}q_k$ boxes with the side length $\delta_k$. For each of these box $\sfB$ we take a $(\delta_k, \ell_{k+1}, \delta_{k+1})$-pattern inside of the box $\sfB$.   
Let $\fC$ be the collections of these boxes, that is 
$$
\fC_{k+1}=\{\sfB_{i_1,\ldots, i_{k+1}}:~1\leqslant i_1\leqslant q_1, \ldots, 1\leqslant i_{k+1}\leqslant q_{k+1}\}.
$$
Our Cantor-like set is defined by  
$$
\FF=\bigcap_{k=1}^{\infty} \FF_k,
$$
where  
$$
\FF_k= \bigcup_{\sfB\in \fC_k} \sfB.
$$ 
There are many possible outcomes by the above construction, we let $\Omega(\fQ; \bdelta, \bell)$ 
denote all  possible patterns.

From our construction clearly we have $\FF_k \supseteq \FF_{k+1}$, and $\FF_k$ is a compact set, and hence $\FF$ is a nonempty compact set. Furthermore we obtain the lower bound of these Cantor sets by using the following  {\it mass distribution
 principle\/}~\cite[Theorem~4.2]{Falconer}.

\begin{lemma} 
\label{lem:mass} 
Let $\cX \subseteq \R^{d}$ and let $\nu$ be a measure on $\R^{d}$ such that $\nu(\cX)>0$. If for any box $B(\vx, r)$ with $0<r\leqslant \varepsilon_0$ for some $\varepsilon_0>0$ we have 
$$
\nu(B(\vx, r))\ll r^{s},
$$
then the Hausdorff dimension of $\cX$ is at least $s$.
\end{lemma}

\begin{lemma}
\label{lem:dimension}  
Let  $\FF\in \Omega(\fQ;\bdelta, \bell)$ and  let $q_{k+1}$, $ k = 0,1, \ldots$, are given 
by~\eqref{eq:qk}.
Then
$$
\dim \FF = \liminf_{k\rightarrow \infty}\frac{\log \prod_{i=1}^{k} q_i}{- \log \delta_k}.
$$
\begin{proof}
We show the  the upper bound of $\dim {\FF}$ first. Let 
$$s>t=\liminf_{k\rightarrow \infty}\frac{\log \prod_{i=1}^{k} q_i}{- \log \delta_k}.$$ 
Then there exists a sequence $k_j$, $j\in   \N$,  such that  
$$
\prod_{i=1}^{k_j} q_i \leqslant \delta_{k_j}^{-s}.
$$
The construction of $\FF$ implies for each $j\in \N$
$$
{\FF}\subseteq \bigcup_{{\sfB}\in \fC_{k_j} } \sf B.
$$
Thus for any  $\varepsilon>0$ we obtain
$$
\sum_{\sfB\in \fC_{k_j}}  (\diam \sfB)^{s+\varepsilon} \ll \delta_{k_j}^{s+\varepsilon}\prod_{i=1}^{k_j} q_i 
\ll\delta_{k_j}^{\varepsilon} \rightarrow 0 \text{ as } j \rightarrow \infty.  
$$  
The definition of Hausdorff dimension, see Definition~\ref{def:Hausdorff}, implies  that  $\dim \FF \leqslant s+\varepsilon$. By the arbitrary choices of $\varepsilon>0$ and $s>t$ we obtain the upper bound 
$$
\dim \FF\leqslant t.
$$

Now we turn to the lower bound of $\dim \FF$.  We first define a measure on $\FF$ (natural measure). For each $k$ and any subset $\mathcal{A}$ let 
$$
\nu_k(\mathcal{A} )= \delta_k^{d} \prod_{i=1}^{k}\frac{1}{q_i} \int {\bf 1}_{\mathcal{A}\cap \FF_k}(x) dx, 
$$
where ${\bf 1}_\cV$ is the indicator function of a set $\cV$.  
Observe that for each ${\sfB}\in \fC_k$ we have 
$$
\nu_k(\sfB)=\prod_{i=1}^{k} q_i^{-1}.
$$ 
We note that the  measure $\nu_k$ weakly convergence to a measure $\nu$, see~\cite[Chapter~1]{Mattila1995}.

Let $\varepsilon>0$ then there exists $k_0$ such that for any $k\geqslant k_0$ we have 
\begin{equation}
\label{eq:lower}
\prod_{i=1}^{k}q_i\geqslant \delta_k^{-t+\varepsilon}. 
\end{equation}
Let $\sfB(\vx, r)\subseteq\Tor$ with $r\leqslant \delta_{k_0}$.  Then there exists $k\geqslant k_0$ such that
$$
\delta_{k+1}\leqslant r \leqslant \delta_k.
$$
Observe that  
$$
\nu(\sfB(\vx, r))\ll \left(\frac{r}{\ell_{k+1}}\right)^{d}  \prod_{i=1}^{k}q_i^{-1}.
$$ 
Applying $q_{k+1}=(\delta_k/\ell_{k+1})^d$, we obtain
$$
\nu(\sfB(\vx, r))\ll \left(\frac{r}{\delta_k}\right)^{d}  \prod_{i=1}^{k}q_i^{-1}.
$$
Combining with the estimate~\eqref{eq:lower} and the condition $\delta_{k+1}\leqslant r\leqslant \delta_k$, we have 
$$
\nu(\sfB(\vx, r))\ll   r^{d} \delta_k^{t-d-\varepsilon}\ll r^{t-\varepsilon}. 
$$
Applying the mass distribution principle  given in Lemma~\ref{lem:mass},  we have $\dim \FF\geqslant t-\varepsilon$. By the arbitrary choice of $\varepsilon>0$ we obtain that $\dim \FF\geqslant t$, which finishes the proof.
\end{proof}
\end{lemma}

\subsection{Monomial exponential sums}

We need the following elementary statement, see, for example~\cite[Equation~(82)]{Kor}  for a more general statement.

\begin{lemma} 
\label{lem:largevalue-mon}
Let  $a\in \Z$ with $\gcd(a, p)=1$, then for integer $d\ge 2$ 
$$
\sum_{n=1}^{p^{d}} \e\left (\frac{a n^{d}}{p^{d}}\right )=p^{d-1}.
$$
\end{lemma}
One can certainly  adapt the arguments in the proof of Lemma~\ref{lem: large sum N}  to get a lower bound 
on $\sigma_d(x; N)$. However we can achieve better results with the following approximate formula 
of Vaughan~\cite[Theorem~4.1]{Vau}. 

\begin{lemma} 
\label{lem:approx-mon}
Let 
$$
x =\frac{a}{q} + \xi
$$ 
with some relatively prime integers $a$ and $q \ge 1$. 
Then 
$$
\sigma_d(x; N) =\frac{1}{q} \sigma_d(a/q; q) \int_0^{N} \e\(\xi \gamma^d\) d \gamma
+ O\(q^{1/2 + o(1)} \(1+ |\xi| N^d\)^{1/2}\). 
$$
\end{lemma}

We now easily see that Lemma~\ref{lem:approx-mon} implies the following  result. 

\begin{lemma}
\label{lem:continuous-mon}
Let  $a\in \Z$ and let $p$ be a prime number such that $\gcd(a, p)=1$.  Let $x\in [0, 1)$ with  $|x-a/p^{d}|<p^{-\tau}$ for some $\tau>0$.  There exists an absolute constant $c> 0$ such 
that for any $\varepsilon >0 $  If 
$$
c p^{\tau/d}\geqslant  N \geqslant  p^{d/2+1 + \varepsilon}  
$$ 
then 
$$
|\sigma_d(x; N)| \geqslant 0.5  N p^{-1},
$$
provided that $p$ is large enough.
\end{lemma}

\begin{proof} 
Using  Lemma~\ref{lem:approx-mon} with $\xi = x-a/p^{d}$ we see that the assumed 
upper bound on $N$ implies that
 $$
 |\xi| N^d = |x-a/p^{d}| N^d<p^{-\tau} N^d \le c^d.
 $$ 
Hence taking $c$ small enough we obtain 
$$
\left|\int_0^{N} \e\(\xi \gamma^d\) d \gamma\right| \ge \frac{2}{3} N.
$$
Therefore by Lemmas~\ref{lem:largevalue-mon} and~\ref{lem:approx-mon}
$$
|\sigma_d(x; N)| \ge  \frac{2}{3 p^d} N  \sigma_d(a/p^d; p^d) + O\(p^{d/2 + o(1)}\)
=   \frac{2}{3} N  p^{-1} + O\(p^{d/2 + o(1)}\) 
$$ 
Recalling the lower bound $N$ we see that the first term dominates, 
which finishes the proof.
\end{proof}


\section{Proofs of abundance  of  large   Weyl sums}

\subsection{Proof of Theorem~\ref{thm:Baire}}

The idea is that we first show that the  exponential sums  $S_d(\vx; N)$ are large at a dense subset of $\Tor$, and then we show the exponential sums are still large at  the small neighbourhoods of these  points.  This implies that the subset $\cE_{\alpha, d}$  has large topology  for each $0<\alpha<1$.

Let the sets $\L_{m, p}$ be  as in  Lemma~\ref{lem: large sum N}. 

For positive integers $k$  and  $m$ we consider the sets
$$
\cG_{m, k}=\bigcup_{\substack {p \geqslant k \\ p\text{ is prime} } } \L_{m, p},
$$
and define 
$$
\cG=\bigcap_{m=1}^{\infty} \bigcap_{k=1}^{\infty} \cG_{m,k}.
$$

Using  Lemma~\ref{lem: large sum N}, with $N = p \fl{c p^{m/d-1} (\log p)^{-1}}$, we conclude that  
 for each $0<\alpha<1$ we have  
\begin{equation}
\label{eq:G in E}
\cG\subseteq\cE_{\alpha, d}.
\end{equation}

Let $m, k\in \N$ and $\sfB\subseteq \Tor$ be an arbitrary open cubic box. Then Corollary~\ref{cor:dense}  implies that  there exists an open cubic box $\widetilde \sfB\subseteq \cG_{m, k}$ such that $\widetilde \sfB\subseteq \sfB$. It follows that  $\Tor\setminus \cG_{m, k}$ is a nowhere dense subset. Furthermore since 
$$
\Tor\setminus \cG= \bigcup_{m=1}^{\infty}\bigcup_{k=1}^{\infty} \left(\Tor\setminus \cG_{m, k}\right),
$$
we obtain that the set $\Tor\setminus \cG$
is the countable union of nowhere dense sets, and hence $\Tor\setminus \cG$ is of first category.
  Together with~\eqref{eq:G in E}  we  complete the proof.

\subsection{Proof of Theorem~\ref{thm:dimension}}
\subsubsection{Preamble} 
We first note that our methods for the cases $d=2$ and $d\geqslant 3$ are different.  For the case $d=2$ we use 
Lemma~\ref{lem:Gauss}. As it is shown in Remark~\ref{rem:why1},  in general   Lemma~\ref{lem:Gauss} does not hold for $d\geqslant 3$, for this case we use the results  from  Section~\ref{sec:weyl sum}. 

Throughout the proof we fix the cube $\fQ$.  In particular, all implied constants may depend on $\fQ$.

We use $\langle  \vz \rangle$ to denote the distance in the $L^\infty$-norm between $\vz \in \R^d$ 
and the closest point $\Z^d$. 

\subsubsection{Case~(i):   $d=2$.} \quad 
 For $\tau>2$ we define  
$$
\cW(\tau)=\{x \in \T:~\langle q x\rangle< q^{1-\tau} \text{ for infinitely many }  q\in \N\}.
$$
The classical   Jarn\'ik--Besicovitch theorem, see~\cite[Theorem~10.3]{Falconer} or~\cite{BDV}, asserts that 
$$
\dim \cW(\tau)=2/\tau.  
$$
We note that the method in the proof of~\cite[Theorem~10.3]{Falconer} (or see the proof of Lemma~\ref{lem:BJ}) imply that   
\begin{equation} 
\begin{split}
\label{eq:JB}
\dim \{x \in \T:~\langle p x\rangle &< p^{1-\tau} \text{ for infinitely primes }  p\}=2/\tau.
\end{split}
\end{equation} 
For our purpose we need obtain an analogy of~\eqref{eq:JB} for 
$\vx \in \fQ \subseteq \T_2$. 

We introduce some notation first. For a prime number $p$ we define
$$
\cA_{\tau, p}=\bigcup_{1\leqslant i, j\leqslant p-1} \left\{\vx\in \fQ:~\left\| \vx -(i/p, j/p) \right\|_{L^\infty} < p^{-\tau} \right\},
$$  
where $\left\| \vz \right\|_{L^\infty} $ is the $L^\infty$-norm in $\R^2$, 
and 
$$
\cG_{\tau}=\bigcap_{k=1}^{\infty}\bigcup_{\substack {p \geqslant k \\ p\text{ is prime} } } \cA_{\tau, p}. 
$$
Applying the arguments of~\cite[Theorem~10.3]{Falconer} to our setting $\cG_{\tau}$ we have the following. 

\begin{lemma}
\label{lem:BJ}
Using the above notation  for any $\tau>2$ we have $$\dim G_{\tau}=3/\tau.$$
\end{lemma}
\begin{proof}  
For the upper bound first note that for each $p$
the set $\cA_{\tau, p}$
can be covered by at most $p^{2}$ boxes with the side length $2p^{-\tau}$.  
Since for each $k \in \N$  
$$
\cG_{\tau}\subseteq \bigcup_{\substack {p \geqslant k \\ p\text{ is prime} } }\cA_{\tau, p},
$$ 
and  for any  $s>3/\tau$ we have
$$
\sum_{\substack {p \geqslant k \\ p\text{ is prime} } } p^{2 -\tau s} \ll k^{3-\tau s} \rightarrow 0 \text{ as } k\rightarrow \infty, 
$$
Definition~\ref{def:Hausdorff} implies $\dim \cG_{\tau}\leqslant s$. By the arbitrary choice of $s>3/\tau$ we conclude
\begin{equation}
\label{eq: G UP}
\dim \cG_{\tau}\leqslant 3/\tau.
\end{equation}

Now we turn to the lower bound. Let $p_k$ be a sequence rapidly increasing prime numbers such that 
\begin{equation}
\label{eq:increase}
p_1\ldots p_k = p_{k+1}^{o(1)}, \qquad \text{as}\ k \to \infty.
\end{equation}
For each $k$ define 
$$
\cH_{k}=\bigcup_{
\substack {p_k\leqslant p\leqslant 2p_k \\ p\text{ is prime} }} \cA_{\tau, p}.
$$
An important fact  is that for different primes $p_k\leqslant p, r\leqslant 2p_k 
$ the sets $\cA_{\tau, p}$ and $\cA_{\tau, r}$ are  disjoint when $p_k$ is large enough. 
Indeed, this follows from   the choice of $\tau>2$ and that for $1\leqslant a, b\leqslant p$ and  $1\leqslant c,d\leqslant r-1 $,  
$$
\left\| (a/p, b/p) -(c/r, d/p) \right\|_{L^\infty}  \gg p_k^{-2}.
$$
Note that there are $p_k^{1+o(1)}$ prime numbers between $p_k$ and $2p_k$, and for each prime number $p_k\leqslant p\leqslant 2p_k$ the set $\cA_{\tau, p}$ contains  $p_k^{2+o(1)}$ boxes with the side length $p_k^{-\tau}$, which due to the fact that  the cube $\fQ$ is fixed.  Thus  the set $\cH_k$ consists of   $p_k^{3+o(1)}$ boxes with the side length $p_k^{-\tau}$. We remark that  the implied constant may depend on $\fQ$, however it is not hard to see that for a fixed   cube $\fQ$ this constant does not affect the result. 
Let 
$$
\cH=\bigcap_{k=1}^{\infty} \cH_{k}.
$$
We claim that 
\begin{equation}
\label{eq:claim}
\dim \cH \geqslant 3/\tau.
\end{equation}
We show some explanation in the following. For each $k\in \N$ let
$$
\FF_k= \bigcap_{i=1}^{k} \cH_{i}. 
$$
Note that $\cH=\bigcap_{k=1}^{\infty}\FF_k$. An important fact is that for any box of $\cH_i$ with the side length $p_i^{-\tau}$ it contains 
$$
q_{i+1}=\left( \frac{p_i^{-\tau}}{p_{i+1}^{-1}}\right)^{3}
$$
uniformly distributed boxes of $\cH_{i+1}$ with the side length $p_{i+1}^{-\tau}$.  Denote $q_1=p_1^{3}$.
It follows, also using~\eqref{eq:increase}, that  $\FF_k$  contains at least  
$$
 \prod_{i=1}^{k}q_i=p_k^{3+o(1)}
$$
boxes with the side length $p_{k}^{-\tau}$.  

By giving a measure on $\cH$ in a similar way as in the proof of Lemma~\ref{lem:dimension}, and then applying the mass distribution principle,  see  Lemma~\ref{lem:mass}, we obtain
$$
\dim \cH \geqslant \liminf_{k\rightarrow \infty} \frac{\log \prod_{i=1}^{k}q_i}{\log p_k^{\tau}}=3/\tau,
$$ 
which proves the claim~\eqref{eq:claim}.

Observe that for each $\vx\in \cH$ there are infinitely $p$ such that $\vx \in \cA_{\tau, p}$,
and hence $\vx\in \cG_{\tau}$ and $\cH\subseteq \cG_{\tau}$.  By the monotonicity  property of  the Hausdorff dimension we 
see from~\eqref{eq:claim} that  
$$
\dim \cG_{\tau}\geqslant \tau/3,
$$
which together with~\eqref{eq: G UP} finishes the proof.
\end{proof}

To conclude the proof  for the case $d=2$, 
it is sufficient to prove $\cG_\tau\subseteq \cE_{\alpha, 2}\(\fQ\)$   with some  $\tau$, since   
\begin{equation}\label{eq:h}
\dim \cE_{\alpha, 2} \(\fQ\) \geqslant \dim \cG_{\tau}\geqslant 3/\tau.
\end{equation}

Let $\vx=(x_1, x_2) \in \cA_{\tau, p}$ then there exists $(a, b)$ with $1\leqslant a,  b\leqslant p-1$ such that
$$
\left\|(x_1, x_2)-(a/p, b/p)\right\|_{L^\infty} <p^{-\tau}.
$$
Applying Lemma~\ref{lem:Gauss}, exactly as in the proof of Lemma~\ref{lem: large sum N} we see that 
$$
\sum_{n=1}^{N} \e(x_1n+x_2n^{2})\gg \frac{N}{\sqrt{p}}, 
$$
provided  
\begin{equation}\label{eq:c1}
p\leqslant N \leqslant c p^{\tau/2} (\log p)^{-1}\mand p \mid N
\end{equation}
for some absolute constant $c> 0$.  

 Furthermore, for any small $\varepsilon>0$, if we have 
\begin{equation}\label{eq:c2}
N/\sqrt{p}\geqslant N^{\alpha+\varepsilon},
\end{equation}
then we  also have
$$
|S_2(\vx, N)|\gg N^{\alpha+\varepsilon}.
$$
Note that the implied constant here does not depend on $\varepsilon$.
Clearly we can find $N$ satisfying~\eqref{eq:c1} and~\eqref{eq:c2} simultaneously provided that 
\begin{equation}
\label{eq:tau d2}
\tau>\max\{2, 1/(1-\alpha-\varepsilon)\}
\end{equation} 
and $p$ is large enough.  It follows that for each $\vx\in A_{\tau, p}$  with large enough $p$ there exists $N=N_p$ such that 
$$
|S_2(\vx; N)| \gg N^{\alpha+\varepsilon}\geqslant N^{\alpha}.
$$
This implies that  $\cG_{\tau}\subseteq \cE_{\alpha, 2}\(\fQ\)$.  Combining with~\eqref{eq:h} and~\eqref{eq:tau d2} we obtain that 
$$
\dim \cE_{\alpha, 2} \(\fQ\) \geqslant \min\left\{3/2, 3(1-\alpha-\varepsilon) \right\}.
$$
By the arbitrary choice of small and positive $\varepsilon$, we finish the proof.

\subsubsection{Case~(ii):   $d\ge3$.} \quad 
We note that our method also works for $d = 2$, thus we only assume $d\ge2$ in the following.

Let $p_k$ be a sequence rapidly increase prime numbers  such that
\begin{equation}
\label{eq:increse2}
p_1\ldots p_k= p_{k+1}^{o(1)}, \quad \text{ as } k\rightarrow \infty.
\end{equation}
Let $\tau>0$ such that  
\begin{equation} 
\label{eq:tau}
\tau>d.
\end{equation} 

As before, we define $\delta_0$ as the side length of $\fQ$, that is, as in~\eqref{eq:delta0}.
 For each $k\in \N$ let 
\begin{equation}
\label{eq:setting1}
 \delta_k=p_{k}^{-\tau}, 
\end{equation}
and choosing  $p_k$  large enough, we see that  we can  assume that $\delta_k < \delta_0$. 

Fix some sufficiently small $\varepsilon>0$ and for each $k\geqslant 0$  let
\begin{equation}
\label{eq:eps}
p_{k+1}^{-\kappa_d+\varepsilon} \leqslant \ell_{k+1}\leqslant 2p_{k+1}^{-\kappa_d+\varepsilon}
\end{equation} 
where $\kappa_d$ is given by~\eqref{eq:kappa}, 
such that $\delta_k/ \ell_{k+1} \in \Z $. For example, the choice   
$$  
\ell_{k+1} = \delta_k /\fl{p_{k+1}^{\kappa_d-\varepsilon}\delta_k}
$$
is satisfactory since we may choose $p_k$ such that   $p_{k+1}^{\kappa_d-\varepsilon} \delta_k \geqslant 1$ for any small  $\varepsilon >0$.

Denote 
\begin{equation}
\label{eq:setting}
 q_{k+1}=\left(\frac{\delta_k}{\ell_{k+1}}\right)^{d}.
\end{equation}
Applying Lemma~\ref{lem:dimension} to the sequences $\delta_k, \ell_k$ we obtain the following.

\begin{lemma}
\label{lem:dim F} 
In the above notation~\eqref{eq:setting1} and~\eqref{eq:setting} and 
under the conditions~\eqref{eq:increse2}, \eqref{eq:tau} and~\eqref{eq:eps},  for any 
$\FF\in \Omega( \fQ; \bdelta, \bell)$, we have  
$$
\dim \FF = \frac{d \kappa_d}{\tau}-d \varepsilon/\tau.
$$
\end{lemma}

\begin{proof}
Recalling~\eqref{eq:increse2} and~\eqref{eq:eps}, we obtain
$$
q_1\ldots q_k= \frac{(p_1\ldots p_k)^{d\kappa_d-d\varepsilon +o(1)}}{(p_1\ldots p_{k-1})^{\tau d}} = p_k^{d\kappa_d-d\varepsilon +o(1)}
$$
and 
$$
\frac{\log q_1\ldots q_k}{\log p_k ^{\tau}}=\frac{d \kappa_d}{\tau} -d\varepsilon/\tau+o(1).
$$
Lemma~\ref{lem:dimension} gives 
$$
\dim \FF =\liminf_{k\rightarrow \infty}\frac{\log q_1\ldots q_k}{\log p_k ^{\tau}}=\frac{d \kappa_d}{\tau} -d\varepsilon/\tau,
$$
which finishes the proof.
\end{proof}

We are now going to show that there exists a pattern $\FF\in \Omega(\fQ; \bdelta, \bell)$ such that $\FF\subseteq \cE_{\alpha, d}\(\fQ\)$ for some $\tau$ which may  depend on $\alpha$ and $d$. Thus Lemma~\ref{lem:dim F} implies that  
\begin{equation}\label{eq:sufficient}
\dim \cE_{\alpha, d}\(\fQ\) \geqslant \dim \FF = \frac{d \kappa_d}{\tau}-d \varepsilon/\tau.
\end{equation}

Our construction is inductive. 

For $\delta_0$ given by~\eqref{eq:delta0} and $\ell_1$
with 
$$
p_1^{-\kappa_d+\varepsilon}\leqslant \ell_1 \leqslant 2p_1^{-\kappa_d+\varepsilon}
$$ 
(note that we  request $ \delta_0/ \ell_1 \in \Z$), by Lemma~\ref{lem:manyboxes}  there exists a $(\delta_0, \ell_1, p_1^{-\tau})$-pattern, which we denote by  $\FF_1$, such that for  
$$ 
N= p_1 \fl{c p_1^{\tau/d-1} (\log p_1)^{-1}}
$$  
and all $\vx \in \FF_1$ we have 
$$
|S_d(\vx; N)| \gg  N^{1-d/2\tau}(\log N)^{-d/2\tau}.
$$

Now, suppose  that we have a pattern $\FF_k$ which is a collection of $q_1\ldots q_k$ boxes with the side length $\delta_k$. For each box $\sfB$ again by  Lemma~\ref{lem:manyboxes} there exists a $(\delta_k, \ell_{k+1}, \delta_{k+1})$-pattern $ \Upsilon_{\sfB} \subseteq \sfB$ such that for  $$
N=p_{k+1}\fl{c p_{k+1}^{\tau/d-1} (\log p_{k+1})^{-1}}
$$ 
and all $\vx \in \Upsilon_{\sfB}$ we have 
\begin{equation}
\label{eq:large}
|S_d(\vx; N)| \gg  N^{1-d/2\tau}(\log N)^{-d/2\tau}.
\end{equation}
Let 
$$
\FF_{k+1}=\{\Upsilon_{\sfB}:~{\sfB}\in \FF_k\}.
$$
For convenience  we use the same notation to denote 
$$
\FF_{k+1}=\bigcup_{\sf B\in \FF_k} \Upsilon_{\sfB}.
$$ 
Let 
$$
\FF=\bigcap_{k=1}^{\infty}\FF_k. 
$$
Then by~\eqref{eq:large} we conclude that 
$$
\FF\subseteq \cE_{\alpha, d} \(\fQ\)
$$ 
provided that 
\begin{equation}
\label{eq:alpha}
1-d/2\tau>\alpha,
\end{equation}
and the condition~\eqref{eq:tau} holds.

The inequalities~\eqref{eq:tau} and~\eqref{eq:alpha} imply that it is sufficient to take  any $\tau$ such that 
$$
\tau>\max\left\{ d, \frac{d}{2(1-\alpha)} \right\}. 
$$ 
Combining this with~\eqref{eq:sufficient}, and using that $d \varepsilon/\tau \leqslant  \varepsilon$ we obtain 
$$
\dim \cE_{\alpha, d}\(\fQ\) \geqslant  \min\left\{ \kappa_d,  2 \kappa_d (1-\alpha) \right\}
-  \varepsilon.
$$  
Since this lower bound holds for any $\varepsilon>0$, we conclude the proof
of Theorem~\ref{thm:dimension}.

\section{Proofs of abundance  of  large monomial sums}

\subsection{Proof of Theorem~\ref{thm:Baire-mon}}

For $d, p \in \N$ and some  $\tau>0$ we define the sets
\begin{equation}
\label{eq:A}
\cA_{d, p, \tau}=\bigcup_{\substack {1\leqslant a <p^{d} \\ \gcd(a, p)=1} } \left \{ x\in \T :~\left | x- a/p^{d}\right | <p^{-\tau} \right \},
\end{equation} 
and 
\begin{equation}
\label{eq:B}
\cB_{d, \tau}=\bigcap_{k=1}^{\infty}\bigcup_{\substack {p \geqslant k \\ p\text{ is prime} } } \cA_{d, p, \tau}. 
\end{equation}

Let  $x\in \cA_{d, p, \tau}$.  Applying  Lemma~\ref{lem:continuous-mon} we see that
$$
\left |\sum_{n=1}^{N} \e(xn^{d}) \right |\geqslant 0.5 Np^{-1}, 
$$
provided that 
\begin{equation}\label{eq:ccc1}
c p^{\tau/d}\geqslant  N \geqslant  p^{d/2+1 + \varepsilon}  
\end{equation}
for some $\varepsilon >0$ and sufficiently large $p$, where $c>0$ is an absolute constant.

Furthermore,  for each $0<\alpha<1$ if we have 
\begin{equation}
\label{eq:ccc2}
0.5 Np^{-1}\geqslant N^{\alpha},
\end{equation}
then we also have 
$$
\left |\sigma_d(x; N) \right |\geqslant N^{\alpha}.
$$
By conditions~\eqref{eq:ccc1} and~\eqref{eq:ccc2} we conclude that for any $\tau>0$ such that  
\begin{equation}
\label{eq:tauvalue}
\tau>\max\{d^{2}/2+d,  d/(1-\alpha)\},
\end{equation}
there exists $N$ such that the conditions~\eqref{eq:ccc1} and~\eqref{eq:ccc2} hold simultaneously. 

 It follows that 
there exists some $N_{d, p, \tau}$ such that for any $x\in \cA_{d, p, \tau}$  
$$
\left| \sigma_d(x; N_{d, p, \tau})\right| \geqslant N_{d, p, \tau}^{\alpha}.
$$
Therefore if~\eqref{eq:tauvalue} holds then 
\begin{equation}
\label{eq:subset}
\cB_{d, \tau}\subseteq\sE_{ \alpha, d}.
\end{equation}

For each $k\in \N$ let
$$
\cG(d, \tau, k)=\bigcup_{\substack {p \geqslant k \\ p\text{ is prime} } } \cA_{d, p, \tau}.
$$
Clearly for each $d, \tau, k$ the set $\cG(d, \tau, k)$ is an open and dense subset of $[0 ,1)$, and hence  $[0, 1)\setminus \cG(d, \tau, k)$ is a nowhere dense subset of $[0 ,1)$. Therefore we obtain that the set 
$$
\bigcup_{k=1}^{\infty} \, [0, 1)\setminus \cG(d, \tau, k)
$$
is of first Baire category set. Now from~\eqref{eq:B} and~\eqref{eq:subset} we obtain 
$$
[0, 1)\setminus \sE_{ \alpha, d} \subseteq  [0, 1) \setminus \cB_{d, \tau} =\bigcup_{k=1}^{\infty} \, [0, 1)\setminus \cG(d, \tau, k),
$$
and hence we finish the proof.

\subsection{Proof of Theorem~\ref{thm:dimension-mon}}  

\subsubsection{Preamble} 
We note that for the monomials the methods for the cases $d=2$ and $d\geqslant 3$ are also different. For the case $d=2$ we use  Lemma~\ref{lem:Gauss}, while for the case $d\geqslant 3$ we use Lemma~\ref{lem:largevalue-mon}.

Throughout the proof we fix the interval $\fI\subseteq \T$.
 In particular, all implied constants may depend on $\fI$.

\subsubsection{Case~(i):   $d=2$.} \quad This case follows by applying the similar arguments to the proof of Theorem \ref{thm:dimension} for the case $d=2$.

For $ p \in \N$ and some  $\tau>0$ let 
$$
\cA_{p, \tau}=\bigcup_{1\leqslant a <p } \left \{ x\in \fI :~\left | x- a/p\right | <p^{-\tau} \right \},
$$ 
and 
$$
\cB_{\tau}=\bigcap_{k=1}^{\infty}\bigcup_{\substack {p \geqslant k \\ p\text{ is prime} } } \cA_{p, \tau}. 
$$
As we claimed before that the method in the proof of~\cite[Theorem~10.3]{Falconer} (or see the proof of Lemma~\ref{lem:BJ}) imply that   
\begin{equation} 
\label{eq:JB-mon}
\dim \cB_{\tau}=2/\tau.
\end{equation}

Applying Lemma \ref{lem:Gauss} and Lemma \ref{lem: large sum N} we conclude that 
for any $x\in \cA_{p, \tau}$ there exists $N_{p, \tau}$ such that 
$$
\sigma_2(x; N_{p, \tau}) \gg N^{\alpha}
$$
provided that   
$$
\tau>\max\{2, 1/(1-\alpha)\}.
$$ 
Note that this is the same condition as~\eqref{eq:tau d2} up to the   small parameter $\varepsilon$. Under this condition for the parameter $\tau$ we conclude 
$\cB_\tau\subseteq \sE_{\alpha, 2}(\fI)$. Combining with~\eqref{eq:JB-mon} we obtain the desired result.

\subsubsection{Case~(ii):   $d\ge3$.} \quad 
We slightly modify  the definition of the set $\cA_{d, p, \tau}$ in~\eqref{eq:A}
by using $\fI$ instead of $\T$, that is, we now set 
$$
\cA_{d, p, \tau}=\bigcup_{\substack {1\leqslant a <p^{d} \\ \gcd(a, p)=1} } \left \{ x\in \fI :~\left | x- a/p^{d}\right | <p^{-\tau} \right \},
$$
while  the set  $\cB_{d, \tau}$ is still defined  by~\eqref{eq:B}. 

By adapting the arguments  of~\cite[Theorem~10.3]{Falconer} and Lemma~\ref{lem:BJ} to the sets $\cB_{d, \tau}$ we have the following.

\begin{lemma}
\label{lem:dimension-mon}
Using the above notation for any $\tau >2d$ we have $$
\dim \cB_{d, \tau}=(d+1)/\tau.
$$
\end{lemma}

\begin{proof} 
Let $s>(d+1)/\tau$. Note that for any $k\in \N$ we have 
$$
\cB_{d, \tau}\subseteq \bigcup_{\substack {p \geqslant k \\ p\text{ is prime} } } \cA_{d, p, \tau}.
$$
Since  
$$
\sum_{p\geqslant k} p^{d} p^{-\tau s}\rightarrow 0\quad  \text{ as } k\rightarrow \infty,
$$
Definition~\ref{def:Hausdorff} implies  $\dim \cB_{d, \tau}\leqslant s$. By the arbitrary choice of $s>(d+1)/\tau$ we conclude that 
\begin{equation} 
\label{eq:dimupper}
\dim \cB_{d, \tau}\leqslant (d+1)/\tau.
\end{equation}

Now we turn to the lower bound of $\dim\cG_{d, \tau}$.  Let $p_k$ be a sequence rapidly increasing prime numbers 
satsifying~\eqref{eq:increase}. 
For each $i$ let 
$$
\cF_k= \bigcup_{\substack { p_k\leqslant p\leqslant 2p_k \\ p\text{ is prime} } } \cA_{d, p, \tau},
$$
and 
$$
\cF=\bigcap_{k=1}^{\infty} \cF_k. 
$$
Clearly we have 
\begin{equation}
\label{eq:FinB}
\cF\subseteq\cB_{d, \tau}
\end{equation}
Hence, it is sufficient to show that 
\begin{equation}
\label{eq:F large}
\dim \cF\geqslant (d+1)/\tau.
\end{equation}

Let $p, q$ be two distinct prime numbers with $p_k\leqslant p, q \leqslant 2p_k$, and  let  $1\leqslant a<p^{d}$ and $ 1\leqslant b< q^{d}$ such that $\gcd(a, p)=\gcd(b, q)=1$. Then  
$$|aq^{d}-bp^{d}|\geqslant 1,$$
and 
$$
\left |\frac{a}{p^{d}}-\frac{b}{q^{d}} \right |\gg \frac{1}{p_k^{2d}}.
$$

Since $\tau>2d$, we conclude that the sets $\cA_{d, p, \tau}$ and $\cA_{d, q, \tau}$ are disjoint for two distinct  prime numbers $p_k\leqslant p, q\leqslant 2p_k$ when $p_k$ is large enough.

Note that there are $p_k^{1+o(1)}$ prime numbers between $p_k$ and $2p_k$, and for each prime number $p_k\leqslant p\leqslant 2p_k$ the set $\cA_{d, p, \tau}$ contains   $p^{d+o(1)}$ intervals  with  length $2p^{-\tau}$ (since the interval $\fI$ is fixed). Thus  the set $\cF_k$ consists of   $p_k^{d+1+o(1)}$ intervals with length nearly $p_k^{-\tau}$.   
As in the proof of Theorem~\ref{thm:dimension},  we remark that  the implied constant may depend on $\fI$, however it is not hard to see that for a fixed   interval  $\fI$ this constant does not affect the result.

By~\eqref{eq:increase}, each interval of $\cF_k$ consists  nearly $p_{k+1}^{d+1+o(1)}$ intervals of $\cF_{k+1}$ of length $p_{k+1}^{-\tau}$.

Applying the method in~\cite[Example~4.7]{Falconer}, see also Lemma~\ref{lem:BJ}, we obtain the 
inequality~\eqref{eq:F large} 
which together with~\eqref{eq:dimupper} and~\eqref{eq:FinB} concludes  the proof.
\end{proof}

 For each $0<\alpha<1$ we intend to find some $\tau>2d$ such that  
$$
\cB_{d, \tau}\subseteq \sE_{ \alpha, d}(\fI).
$$
Hence, by the monotonicity  property of  the Hausdorff dimension and Lemma~\ref{lem:dimension-mon} we obtain 
\begin{equation}
\label{eq:lowerbound}
\dim\sE_{ \alpha, d}(\fI)\geqslant \dim \cB_{d, \tau } =(d+1)/\tau.
\end{equation}

Applying the arguments in the proof of Theorem~\ref{thm:Baire-mon}, see~\eqref{eq:tauvalue},  we obtain that 
for any 
$$
\tau>\max\{d^{2}/2+d,  d/(1-\alpha)\}>2d,
$$ 
and any $ \cA_{d, p, \tau}$ there exists some $N_{d, p, \tau}$ such that for any $x\in \cA_{d, p, \tau}$  
$$
\left| \sigma_d(x; N_{d, p, \tau})\right| \geqslant N_{d, p, \tau}^{\alpha}.
$$
Thus the condition of Lemma~\ref{lem:dimension-mon} is satisfied. Combining with 
~\eqref{eq:lowerbound},  we obtain 
$$
\dim \sE_{ \alpha, d}(\fI)\geqslant  (1+1/d) \min\left \{2/(d +2),  1-\alpha \right \}
$$
which finishes the proof.

\section{Proofs of abundance of poorly distributed polynomials}

\subsection{Exponential sums and the discrepancy} 

For our applications we need the following  {\it Koksma-Hlawlka inequality\/}, see~\cite[Theorem~1.14]{DrTi} for a general statement. 
\begin{lemma}
\label{lem:KH}
Using the above notation, for any $\vx\in \Tor$
$$
S_d(\vx; N)\ll D_d(\vx; N).
$$
\end{lemma}

Note that in particular,  Lemma~\ref{lem:KH} 
implies $\sigma_d(x; N)\ll \Delta_d(x; N)$ for $x \in [0,1)$. 

\subsection{Proof of Theorems~\ref{thm:BaireD} and~\ref{thm:BaireD-mon}}

We see that Lemma~\ref{lem:KH} implies that for any cube $\fQ\subseteq \T_d$ or interval $\fI \subseteq \T$  and any
   $\eps>0$  one has 
\begin{equation} 
\label{eq:subset-cube}
\cE_{\alpha+\eps, d}(\fQ)\subseteq \cD_{\alpha, d}(\fQ) \mand 
\sE_{\alpha+\eps, d}(\fI)\subseteq \sD_{\alpha, d}(\fI). 
\end{equation}
Combining~\eqref{eq:subset-cube} for $\fQ=\T_d$ and $\fI =\T$ with  Theorems~\ref{thm:Baire}  and~\ref{thm:Baire-mon} we obtain 
Theorems~\ref{thm:BaireD} and~\ref{thm:BaireD-mon}, respectively. 

\subsection{Proof of Theorems~\ref{thm:dimD} and~\ref{thm:dimD-mon}}   
Applying~\eqref{eq:subset-cube} and the monotonicity property of Hausdorff dimension we have
$$
\dim \cD_{\alpha, d}(\fQ) \ge \sup_{\eps>0 } \dim \cE_{\alpha+\epsilon, d}(\fQ)
$$
and
$$
\dim \sD_{\alpha, d}(\fI) \ge \sup_{\eps>0 } \dim \sE_{\alpha+\epsilon, d}(\fI).
$$
Combining this with Theorems~\ref{thm:dimension} and~\ref{thm:dimension-mon} 
 we obtain Theorems~\ref{thm:dimD} and~\ref{thm:dimD-mon}, respectively.

\section{Further results, open problems and conjectures}

\subsection{Further extensions of  Theorems~\ref{thm:Baire} and~\ref{thm:dimension}}

On the other hand, the method of proof of Lemma~\ref{lem:box dense} is quite robust and can be implies to some other families of 
polynomials, such as sparse polynomials   
$$
a_1 X^{m_1}+ \ldots + a_d X^{m_d} \in \F_p[X]. 
$$
In turn, this can be used to obtain 
 versions of Theorems~\ref{thm:Baire}  and~\ref{thm:dimension}  for exponential sum with sparse polynomials
$$
S_{\vm}(\vx; N)=\sum_{n=1}^{N}\e(x_1n^{m_1}+\ldots+x_d n^{m_d}), 
$$
where $\vm=(m_1, \ldots, m_d)\in \Z^d$  with $1\leqslant m_1<m_2<\ldots<m_d$.  More precisely,  for each $0<\alpha<1$ and  $\vm=(m_1, \ldots, m_d)\in \Z^d$, we define
$$
\cE_{\alpha, \vm}=\{\vx\in \Tor:~|S_{\vm}(\vx; N)|\geqslant N^{\alpha} \text{ for infinitely many } N\in \N\}.
$$
We note that~\eqref{eq:Mord} can easily be extended to sparse polynomials
\begin{equation}  
\label{eq:Mord-sparse}
\sum_ {\va \in \F_p^d} \left| \sum_{n=1}^{p}\ep\(a_1 n^{m_1}+\ldots +a_d n^{m_d} \)\right|^{2d} \ge  d! p^{2d} + O(p^{2d-1}),
\end{equation}
 which in turn leads to full analogues of 
Lemmas~\ref{lem:LB Dens}, \ref{lem:k}  and~\ref{lem:box dense}.  
Then we have the following direct generalisations of Theorems~\ref{thm:Baire} and~\ref{thm:dimension}
which can be obtained at the cost of essentially only typographical changes in their proofs.
For each $0<\alpha<1$ and  $\vm=(m_1, \ldots, m_d)\in \Z^n$ with $1\leqslant m_1<m_2<\ldots<m_d$, 
\begin{itemize}
\item[\bf{(A)}] \ the subset $\Tor \setminus \cE_{\alpha, \vm}$ is of  the first Baire   category; 

\item[\bf{(B)}]  \  for any cube  $\fQ\subseteq \T_d$ we have,  
$$
\dim \cE_{\alpha, \vm}(\fQ) \geqslant \min\left\{ \frac{d \kappa_d}{m_d}, \frac{2d \kappa_d(1-\alpha)}{m_d} \right\}, 
$$
\end{itemize} 
where  $\kappa_d$ is given by~\eqref{eq:kappa}. Note that 
we recover the bound of  Theorem~\ref{thm:dimension} (for $d\geqslant 3$)  provided $m_d=d$.

\begin{remark} We note that~\eqref{eq:Mord-sparse} is only a lower bound rather than an
asymptotic formula as~\eqref{eq:Mord}.  In fact, most likely an asymptotic form 
of~\eqref{eq:Mord-sparse}  holds with $m_1\ldots m_d$ instead of $d!$, see~\cite{Wool1}. However this is inconsequential for our results.
\end{remark} 

We note that it is natural to try to improve  Theorem~\ref{thm:dimension} via an appropriate 
version of Lemma~\ref{lem:approx-mon} for arbitrary polynomials. Unfortunately the only known 
result in this direction~\cite[Theorem~7.2]{Vau} is not strong enough to lead to such an 
improvement. 

%

\subsection{Further questions about the structure of  Weyl sums}  
For $\vx\in \T$ we  now define   
\begin{align*}
\sigma(\vx)&=\inf\{s>0:~S_d(\vx; N)\ll N^{s}\}\\
&=\sup\{s>0:~|S_d(\vx; N)|\gg N^{s} \text{ for infinitely many } N\in \N\}\\ 
&=\sup\{s>0:~|S_d(\vx; N)|\geqslant N^{s} \text{ for infinitely many } N\in \N\}. 
\end{align*}
Alternatively, we may also define 
\begin{equation}\label{eq:sigma}
\sigma(\vx)=\limsup_{N\rightarrow\infty}\frac{\log |S_d(\vx; N)|}{\log N}.
\end{equation}

By the definition we have   
$$\cE_{\alpha, d} \subseteq \{\vx\in \Tor:~\sigma(\vx)\geqslant \alpha\}.
$$

For each $0\leqslant \alpha\leqslant 1$ we define the level set
$$
\Omega_\alpha=\{\vx\in \Tor:~\sigma(\vx)=\alpha\}.
$$
Clearly these sets $\Omega_\alpha$ form a decomposition of $\Tor$.  There are several
natural questions about these sets.  Note that Conjecture~\ref{conj:1/2} asserts 
that for any  $\alpha\in (1/2, 1]$ we have $\lambda (\Omega_\alpha)=0$. We may make the
following stronger conjecture.

\begin{conj}
\label{conj:lambda_Omega}
For $\alpha\in [0, 1]$ we have 
$$
\lambda(\Omega_\alpha) =  
  \begin{cases}
   0  &  \text{for} \ \alpha \neq 1/2, \\
   1 &  \text{for} \ \alpha =1/2.
  \end{cases}
$$
\end{conj}
 
We may also use the Hausdorff dimension to measure the size of $\Omega_\alpha$.

\begin{question}
 What is the Hausdorff dimension $\dim \Omega_\alpha$ of $\Omega_\alpha$?  
\end{question}

Finally, one can also ask whether the function $\sigma(\vx)$ which  is defined by~\eqref{eq:sigma} has multifractal structure. More precisely  we ask the following: 

\begin{question}
Does there exist a set $\cA\subseteq [0,1]$ with $\lambda(\cA)>0$  such that for any $\alpha \in \cA$  we have 
$$
\dim \Omega_\alpha >0?
$$
\end{question}

\subsection{Further questions about the distribution of large complete rational  sums and possible improvements of  Theorem~\ref{thm:dimension} }

It is certainly natural to consider more general transformations 
\begin{equation}
\label{eq:trans}
f(X) \mapsto f(\lambda X + \mu), \qquad (\lambda,\mu) \in \F_p^* \times \F_p, 
\end{equation}
instead of just $f(X) \mapsto f(\lambda X)$ which is essentially used in the proof of Lemma~\ref{lem:box dense}. 
The transformation~\eqref{eq:trans}  is very similar to the transformation  $f(X) \mapsto \lambda^{-d}f(\lambda X + \mu)$ used in 
 the proof of~\cite[Lemma~4]{Shp}. However, while in~\cite{Shp} the Deligne bound (see~\cite[Section~11.11]{IwKow})
 is applied to the corresponding double exponential sums with a polynomials in  $\lambda$ and $\mu$, in the case of~\eqref{eq:trans} these polynomials are singular, and so the Deligne bound
 does not apply. It is certainly interesting to find an alternative way, and thus improve  Lemma~\ref{lem:box dense}, in which $\kappa_d$ 
 can possibly be replaced with $1/d$.

Lemma~\ref{lem:k} study the distribution of sets
$$
\{(\lambda a_1, \ldots, \lambda^{d}a_d):~\lambda\in \F_p^{*}\}, 
$$
where $a_j\in \F_p^{*}$ for each $j=1, \ldots, d$.  Lemma~\ref{lem:k} asserts that for any box $\fB$ of $\F_p^{d}$  with the side length $L\geqslant C p^{1-1/2d} \log p$ for some large constant $C$ there exists $\lambda \in \F_p^{*}$ such that 
$$
(\lambda a_1, \ldots, \lambda^{d}a_d)\in \fB.
$$
Note that there are totally  $p-1$ vectors 
$$ 
(\lambda a_1, \ldots, \lambda^{d}a_d), \qquad  \lambda\in \F_p^{*},
$$ 
thus the  smallest  $L$ in Lemma~\ref{lem:k} should be  
$$
L\gg p^{1-1/d}.
$$
One could ask that is this a sufficient condition.  

\begin{question}
Let $(a_1, \ldots, a_d)\in (\F_p^{*})^{d}$. Is it true that for any $\varepsilon>0$ there exists a constant $C_\varepsilon$ such that any box $\fB$ of $\F_p^{d} $ with the side length $L\geqslant C_\varepsilon p^{1-1/d+\varepsilon}$  contains a vector $(\lambda a_1, \ldots, \lambda^{d}a_d)$ for some $\lambda \in \F_p^{*}$?
\end{question}

It is also interesting to consider the special case that is the distribution of 
$$
\left\{(\lambda, \lambda^{2}):~\lambda \in \F_p^{*}\right\}. 
$$
Note that studying the distribution of 
$$
\{\lambda^{2}:~\lambda\in \F_p^{*}\}
$$
is already  an interesting and hard problem related to the distribution of  quadratic nonresidues.

A possible approach to improving  Theorem~\ref{thm:dimension}  is via finding an asymptotic formula or at least a lower bound 
for the average of $T_{d, p}(\va)$ over small box $\fB$ as in~\eqref{eq:box}. In fact finding lower bounds for the moments 
$$
\sM_{\nu,d}(\fB) = \sum_{\substack{\va\in \fB\\\va \ne \0}} |T_{d, p}(\va)|^{2\nu}, \qquad \nu =1,  \ldots, d, 
$$
of nontrivial sums with $\va \ne \0$  
is of independent interest.  For $\fB=\F_p^d$ one can easily extend the result of Mordell~\cite{Mord}, 
that is,~\eqref{eq:Mord}, to any $\nu =1,  \ldots, d$ and obtain 
\begin{equation}  
\label{eq: MFq Ad}
\sM_{\nu,d}(\F_p^d) =  A_d(\nu)  p^{d+\nu} + O\( p^{d+\nu-1}\),  
\end{equation}
where 
$$
A_d(\nu)= 
\begin{cases}
d!-1, & \text{for} \ \nu = d,\\
\nu! ,   & \text{for} \ \nu =1,  \ldots, d-1,
\end{cases}
$$
see also~\cite[Equation~(2)]{KnSok2}.

Using the same arguments as in the proof of 
Lemmas~\ref{lem:k} and~\ref{lem:box dense} with $k=d$, one can obtain an
asymptotic formula 
\begin{equation}  
\label{eq: MB MFq}
\sM_{\nu,d}(\fB)   =  A_d(\nu)  L^d p^\nu + O\(p^{d+\nu-1} L^{1/2}( \log p)^{d-1}\),  
\end{equation}
see Appendix~\ref{app:MB MFq},
which is nontrivial in the case of cubes with the side length $L \geqslant p^{2(d-1)/(2d-1)+ \varepsilon}$ for any 
fixed $\varepsilon > 0$. 
 However we are interested in much smaller boxes, for example of size of the side length 
about  $L \sim p^{1/2 + \varepsilon}$. In fact, a lower bound of the form $L^d p^{\nu +o(1)}$ 
for any fixed $\nu$ is sufficient for our applications. 
%

\subsection{An approach to Conjecture~\ref{conj:1/2}} 

Recall that Conjecture~\ref{conj:1/2} asserts that  $\vartheta_d=1/2$ for each integer $d\geqslant 2$,  and the 
 bound~\eqref{eq:M-R} gives $\vartheta_d\leqslant1/2$.  Thus  it is sufficient to prove that for any $0<\alpha<1/2$ one has 
$\lambda(\cE_{\alpha,d})>0$.

For  $0<\alpha<1/2$ and integer $d\geqslant 2$ we define  
$$
\cA_{d, \alpha, i}=\{\vx\in \Tor:~|S(\vx; i)|\geqslant i^{\alpha}\}.
$$ 
We can write 
$$
\cE_{\alpha,d} =\bigcap_{k=1}^{\infty}\bigcup_{i=k}^{\infty} \cA_{\alpha, d, i}.
$$

\begin{lemma}
Let $0<\alpha<1/2$ then  $\lambda(\cA_{d, \alpha, i})\gg 1/i$, and  hence 
\begin{equation}
\label{eq:divergence}
\sum_{i=1}^{\infty} \lambda(\cA_{d, \alpha, i})=\infty.
\end{equation}
\end{lemma}
\begin{proof}
Applying the trivial bound $|S(\vx; i)|\leqslant i$ we obtain 
\begin{align*}
\int_{\Tor} |S(\vx; i)|^{2} d\vx&=\int_{\cA_{\alpha, d, i}} |S(\vx; i)|^{2} d\vx+\int_{\Tor \setminus \cA_{\alpha, d, i}}|S(\vx; i)|^{2} d\vx\\
&\leqslant i^{2} \lambda(E_i)+i^{2\alpha}.
\end{align*}
Combining with Parseval identity 
$$
\int_{\Tor} |S(\vx; i)|^{2} d\vx =i  
$$
and the condition $0<\alpha<1/2$, we obtain the desired result.
\end{proof}

Suppose that  the sets $\cA_{\alpha, d, i}$ are pair independent with respect to the  Lebesgue measure $\lambda$, i.e., for any $i\neq j$ we  have 
$$
\lambda(\cA_{\alpha, d, i} \cap \cA_{\alpha, d, j})=\lambda(\cA_{\alpha, d, i}) \lambda(\cA_{\alpha, d, j}),
$$
then the {\it Borel-Cantelli lemma\/} and~\eqref{eq:divergence} implies that $\lambda(\cE_{\alpha, d})=1$. Surely the  pair  independent  assumption is not true,  and an ordinary way to overcome this is by the following arguments. One first show that these sets are weak independent, that is there exists some constant $C>0$ such that for any $i\neq j$ we have 
$$
\lambda(\cA_{\alpha, d, i} \cap \cA_{\alpha, d, j})\leqslant C\lambda(\cA_{\alpha, d, i}) \lambda(\cA_{\alpha, d, j}),
$$
then a variant of the Borel-Cantelli lemma gives  
$$
\lambda(\cE_{\alpha, d})\geqslant 1/C>0.
$$
Secondly one may use a zero-one law to pass from $\lambda(\cE_{\alpha, d})>0$ to $\lambda(\cE_{\alpha, d})=1$.


%
%

\appendix

\section{Proof of the bound~\eqref{eq:M-R} and some extensions}   
\label{app:A}

%
%
%
%
%
%
%

%

By  applying a very special case of the {\it Menshov--Rademacher  theorem\/}, see~\cite[p. 251]{KaSa} for the general statement,  we conclude that   if for some sequence $c_n, n\in \N$ of complex numbers
we have 
\begin{equation}
\label{eq:converge}
\sum_{n=1}^\infty |c_n|^2 (\log n)^{2}<\infty,
\end{equation}
then the series 
$$
\sum_{n=1}^\infty c_n \e(nx)
$$ 
converges for almost all $x \in [0,1)$. 

For $\vx=(x_1, \ldots, x_d)\in \Tor$ we have 
$$
\e(x_1n+\ldots+x_dn^d)=\e(x_1n)\e(x_2n^2+\ldots+x_dn^d).
$$
It follows that for any $(x_2, \ldots, x_d)\in \T_{d-1}$  the series 
$$
\sum_{n=1}^\infty c_n \e(x_2n^2+\ldots+x_dn^d) \e(x_1n)
$$
converges for almost all $x_1\in [0,1)$. Together with the 
{\it Fubini theorem\/},  we obtain 
that the series 
$$
\sum_{n=1}^\infty c_n \e(x_1n+\ldots+x_dn^d) 
$$
converges for almost all $\vx\in \Tor$.

%

Now we turn to the proof of~\eqref{eq:M-R}. 
We denote 
$$
\log^+ k = \max\{1, \log k\},
$$
and 
$$
\varphi_n(\vx)=\e(x_1n+\ldots+x_dn^d).
$$
Fix any $\gamma> 3/2$,  
and write 
$$
S_d(\vx; N)=\sum_{n=1}^N n^{-1/2}  (\log^+n)^{-\gamma} \varphi_n(\vx) n^{1/2}  (\log^+n)^{\gamma} .
$$
Then the summation by parts gives 
\begin{equation}
\begin{split}
\label{eq:sum}
S_d&(\vx; N)\\&=s_d(\vx;N) N^{1/2}  (\log^+ N)^{\gamma}\\
& \quad +\sum_{k=1}^{N-1}s_d(\vx;k) \(k^{1/2}( \log^+k )^{\gamma}  -(k+1)^{1/2}  (\log^+(k+1))^{\gamma}\), 
\end{split}
\end{equation}
where 
$$
s_d(\vx;k)=\sum_{n=1}^kn^{-1/2}  (\log^+n)^{-\gamma}\varphi_n(\vx).
$$
Since the condition~\eqref{eq:converge} is satisfied, 
for almost all $\vx\in \Tor$ there exits some positive   $B_{\vx}$ such that  for all $k\in \N$ we have
\begin{equation}
\label{eq:Bx}
|s_d(\vx;k)| \leqslant B_{\vx}.
\end{equation}
Substituting~\eqref{eq:Bx} in \eqref{eq:sum} we easily conclude that for almost all $\vx\in [0,1)$  
we have~\eqref{eq:M-R}.

We note that the above arguments implies that for any $(x_2, \ldots, x_d)\in \T_{d-1}$ the bound 
$$
\left| \sum_{n=1}^{N} \e\(x_1 n+\ldots + x_dn^d \) \right| \leqslant  N^{1/2} (\log N)^{3/2+o(1)}
$$
holds for almost all $x_1\in [0,1)$.

Furthermore, one can easily see that the above argument work in a much broader generality. 
For example, let $f_1, \ldots, f_d$ be $d$ functions such that for any $n\in \N$ we have $f_i(n)\in \Z$ 
for each $i=1, \ldots, d$. If one of these functions is  eventually strictly monotonic, then for almost all $(x_1,\ldots, x_d)\in \Tor$ we have
$$
\left| \sum_{n=1}^N \e\(x_1 f_{1}(n)+\ldots+x_d f_d(n)\) \right| \leqslant  N^{1/2} (\log N)^{3/2+o(1)}.
$$
For instance, for $0<t<\infty, a>1$ the bound 
$$
\left| \sum_{n=1}^N \e\(x_1 \fl{n^t}+ x_2 \fl{a^n} +x_3 \fl{\log n}\) \right| \leqslant  N^{1/2} (\log N)^{3/2+o(1)}
$$
holds for almost all $(x_1, x_2, x_3)\in \T_3$.

\begin{remark} 
For the case $d=2$ we can obtain the bound  $N^{1/2}\log N$ for the estimate \eqref{eq:M-R}  in a  different way.  
The {\it Khinchine theorem}, see~\cite[Introduction]{BDV},  implies that  for almost all irrational $x\in [0, 1)$  there 
exits some positive constant $c(x)$ such that for all rational $a/q$ with $\gcd(a, q)=1$ we have 
$$
\left|x-\frac{a}{q} \right|\geqslant \frac{c_x}{(q \log q)^{2}}.
$$
On the other hand, by~\cite[Theorem~8.1]{IwKow},  if $|x-a/q|\leqslant 1/qN$ with $\gcd(a, q)=1$ and  $1\leqslant q \leqslant N$ then for any $y\in [0, 1)$  one has 
$$
\sum_{n=1}^N \e(yn+x n^2) \ll N/q^{1/2}+q^{1/2}\log q. 
$$
Combining these two results, 
we conclude that for almost all $\vx \in \T_2$ one has 
$$
S_2(\vx; N)\ll N^{1/2}\log N.
$$
\end{remark}



\section{Moments of rational exponential sums over small boxes} 
\label{app:MB MFq}

Here we  sketch  a proof of~\eqref{eq: MB MFq}. Clearly we can assume that  
$0 \not \in \cI_1$ (it is easy see that by Lemma~\ref{lem:Weil} discarding $O(p^{d-1})$ 
such sums changes 
the value of  $\sM_{\nu,d}(\fB)$ by $O\(L^{d-1} p^{\nu}\)$, which 
can be absorbed in the error in~\eqref{eq: MB MFq}). In particular, we can assume  that $\0 \not \in \fB$.

Observe that for any $\lambda\in \F_p^{*}$ and $\vb \in \F_p^{d}$  we have 
$$
T_{d, p}(\vb)=T_{d, p}(\lambda\circ \vb), 
$$
where 
$$
\lambda \circ \vb =\left( \lambda b_1, \ldots, \lambda^{d} b_d\right).
$$
It follows that 
\begin{equation} \label{eq:T}
\begin{split}
\sM_{\nu,d}(\fB)&=\frac{1}{p-1}\sum_{\lambda\in \F_p^{*}}\sum_{\vb\in \fB}|T_{d, p}(\lambda \circ\vb)|^{2d}\\
& =\frac{1}{p-1} \sum_{\substack{\va\in \F_p^{d}\\\va \ne \0}}
 N(\va)|T_{d, p}(\va)|^{2\nu},
\end{split}
\end{equation} 
where
$$
N(\va)=\#\{(\lambda, \vb)\in \F_p^{*}\times \fB: \lambda\circ \vb =\va\}.
$$
Let $\Lambda(a)$ be the set of $\lambda\in \F_p^{*}$ with $a\lambda \in \cI_1$
where  $\fB$ is as in~\eqref{eq:box}.
Hence for $\va = (a_1, \ldots, a_d)$, we have 
$$
N(\va)= \#\{\lambda \in \Lambda(a_1):~ \left( \lambda^2 a_2, \ldots, \lambda^{d} a_d\right) \in \cI_2\times \ldots \times \cI_d\}.
$$
By the orthogonality of characters, and then changing the order of summation 
and separating the contribution from $h_2=\ldots = h_d$ we obtain    \begin{equation} \label{eq:N R}
\begin{split}
N(\va)&=\frac{1}{p^{d-1}}\sum_{\lambda \in \Lambda(a_1) } \sum_{y_2\in \cI_2} \ldots \sum_{y_d \in \cI_d}\\
& \qquad \qquad \qquad
\sum_{h_2, \ldots, h_d = -(p-1)/2}^{(p-1)/2} \ep\(\sum_{j=2}^d h_j( \lambda^j a_j - y_j)\)\\
&= \frac{\# \Lambda(a_1) L^{d-1}  }{p^{d-1}} + R(\va), 
\end{split}
\end{equation}
where 
\begin{align*}
R(\va)= \frac{1}{p^{d-1}}  \sum_{\substack{h_2, \ldots, h_d = -(p-1)/2\\(h_2, \ldots, h_d) \ne \0}}^{(p-1)/2} 
& \prod_{i=2}^d
\left| \sum_{y_j\in \cI_i} \ep\(h_j y_j\) \right|\\
& \qquad \qquad \quad 
\left|  \sum_{\lambda \in \Lambda(a_1) }
\ep\(\sum_{j=2}^d   h_j \lambda^j a_j \) \right| .
\end{align*}
We note that $N(\va)=0$ if the first coordinate  of $\va$ is zero. Combining~\eqref{eq:T} and~\eqref{eq:N R},  we obtain
$$
\sM_{\nu,d}(\fB)= \frac{ L^{d-1} }{(p-1)p^{d-1}}
\sum_{\substack{\va\in \F_p^{d}\\ a_1 \ne 0}}\# \Lambda(a_1) |T_{d, p}(\va)|^{2\nu} 
+ O\(W\) , 
$$
where 
$$
W = \frac{1}{p-1}\sum_{a_1\neq 0} |R(\va)|T_{d, p}(\va)|^{2\nu}. 
$$
By Lemma~\ref{lem:Weil} we obtain 
\begin{align*}
\frac{ L^{d-1} }{(p-1)p^{d-1}}
\sum_{\substack{\va\in \F_p^{d}\\ a_1 \ne 0}}&\# \Lambda(a_1) |T_{d, p}(\va)|^{2\nu}
= \frac{ L^{d} }{(p-1)p^{d-1}}
\sum_{\substack{\va\in \F_p^{d}\\ a_1 \ne 0}} |T_{d, p}(\va)|^{2\nu}\\
& =  \frac{ L^{d} }{(p-1)p^{d-1}} \sM_{\nu,d}(\F_p^d)+O\(L^dp^{\nu-1}\)\\ 
& =  \frac{ L^{d} }{p^{d}} \sM_{\nu,d}(\F_p^d)+O\(  L^d p^{-d-1}\sM_{\nu,d}(\F_p^d) + L^dp^{\nu-1}\).
\end{align*} 
Hence, recalling~\eqref{eq: MFq Ad} we obtain 
\begin{equation} \label{eq:MMW}
\sM_{\nu,d}(\fB)= A_d(\nu)   L^{d}    p^{\nu} +O\(L^dp^{\nu-1} + W\).
\end{equation}

To estimate $W$ we note that by~\cite[Equation~(8.6)]{IwKow} we have
$$
R(\va)\ll  \sum_{\substack{h_2, \ldots, h_d = -(p-1)/2\\(h_2, \ldots, h_d) \ne \0}}^{(p-1)/2} 
\prod_{i=2}^d  \min \left\{\frac{L}{p},\frac{1}{|h_i|} \right\}
\left|  \sum_{\lambda \in \Lambda(a_1) }
\ep\(\sum_{j=2}^d  a_j h_j \lambda^j \) \right| .
$$
By  Lemma~\ref{lem:Weil} we now see that 
\begin{align*}
W \ll  p^{\nu -1}\sum_{a_1\in \F_p^*}  \sum_{\substack{h_2, \ldots, h_d = -(p-1)/2\\(h_2, \ldots, h_d) \ne \0}}^{(p-1)/2} &
\prod_{i=2}^d \min \left\{\frac{L}{p},\frac{1}{|h_i|} \right\}
\\
&
\sum_{a_2, \ldots, a_d \in \F_p}  \left|  \sum_{\lambda \in \Lambda(a_1) }
\ep\(\sum_{j=2}^d  a_j h_j \lambda^j \) \right| .
\end{align*}

Using the Cauchy inequality, as in the proof of Lemma~\ref{lem:box dense},  we have
\begin{align*}
&\(\sum_{a_2, \ldots, a_d \in \F_p}  \left|  \sum_{\lambda \in \Lambda(a_1) }
\ep\(\sum_{j=2}^d  a_j h_j \lambda^j \) \right| \)^2\\
& \qquad\qquad\quad  \leqslant p^{d-1}  \sum_{a_2, \ldots, a_d \in \F_p}  \left|  \sum_{\lambda \in \Lambda(a_1) }
\ep\(\sum_{j=2}^d  a_j h_j \lambda^j \) \right|^2\ll p^{2(d-1)}  L. 
\end{align*}
Hence,
\begin{align*}
W & \ll  p^{\nu +d -2}L^{1/2} \sum_{a_1\in \F_p^*}  \sum_{\substack{h_2, \ldots, h_d = -(p-1)/2\\(h_2, \ldots, h_d) \ne \0}}^{(p-1)/2} 
\prod_{i=2}^d \min \left\{\frac{L}{p},\frac{1}{|h_i|} \right\}\\
&\leqslant  p^{\nu +d -1}L^{1/2}   \sum_{\substack{h_2, \ldots, h_d = -(p-1)/2\\(h_2, \ldots, h_d) \ne \0}}^{(p-1)/2} 
\prod_{i=2}^d \min \left\{\frac{L}{p},\frac{1}{|h_i|} \right\} \\
& \ll  p^{\nu +d -1}L^{1/2} (\log p)^{d-1}, 
\end{align*}
which together with~\eqref{eq:MMW} yields~\eqref{eq: MB MFq}.

\section*{Acknowledgement}

The authors are grateful to Fernando Chamizo,  Boris Kashin, Bryce Kerr, Sergei Konyagin and Trevor Wooley for  helpful advice and discussions.

This work was  supported in part  by ARC Grant~DP170100786.

\end{document}